\documentclass[journal,twoside,web]{IEEEtran}
\usepackage{cite}
\usepackage{amsmath,amssymb,amsfonts}
\usepackage{amsthm}
\usepackage{algorithmic}
\usepackage{graphicx}
\usepackage{graphics}
\usepackage{tcolorbox}
\usepackage{url}
\usepackage{color}
\usepackage{hyperref}
\usepackage{textcomp}
\usepackage{mathtools, cuted}
\usepackage{lipsum}
\usepackage{stfloats}
\usepackage{multirow}
\usepackage{wrapfig}
\usepackage{graphicx}
\usepackage{caption}
\usepackage{subcaption}
\usepackage[T1]{fontenc}
\usepackage[utf8]{inputenc}

\definecolor{subsectioncolor}{RGB}{0,0,128} % Dark blue

\usepackage{amsmath,amssymb,euscript,psfrag,latexsym,graphicx}
\usepackage{bbm,color,amstext,wasysym,cuted,mathtools, cite}

\usepackage[normalem]{ulem}
\graphicspath{{./},{./figures/}}
\usepackage{dsfont}
\usepackage{ulem}

\newcommand{\cJ}{{\cal J}}

\newcommand{\cC}{{\cal C}}
\newcommand{\cK}{{\cal K}}
\newcommand{\cF}{{\cal F}}

\newcommand{\cH}{{\cal H}}
\newcommand{\cV}{{\cal V}}

\newcommand{\cD}{{\cal D}}

\newcommand{\mC}{{\mathbb C}}

\newcommand{\mR}{{\mathbb R}}
\usepackage{algorithm}
\usepackage{algorithmic}
\usepackage[algo2e,linesnumbered,ruled]{algorithm2e}

\newcommand{\mU}{{\mathbb U}}

	\usepackage[utf8]{inputenc}

\newcommand{\bk}{{\mathbf k}}

\usepackage{mathtools,halloweenmath}

\newtheorem{theorem}{Theorem}
\newtheorem{definition}{Definition}
\newtheorem{lemma}{Lemma}
\newtheorem{assumption}{Assumption}
\newtheorem{remark}[theorem]{Remark}
\newtheorem{property}{Property}

\usepackage[dvipsnames]{xcolor}

\newtheorem{proposition}{Proposition}
\begin{document}

\title{Discovering the Kalman-Bucy-Koopman Filter}

\author{ Umesh Vaidya, {\it IEEE Senior Member}
\thanks{
Financial support from of NSF CMMI 2031573 is greatly acknowledged. The author is with the Department of Mechanical Engineering, Clemson University, Clemson SC, 29631. }}

\maketitle

\begin{abstract}
This paper introduces the Kalman Bucy Koopman (KBK) filter, a novel framework for nonlinear state estimation grounded in Koopman operator spectral theory. The nonlinear estimation problem is formulated as a maximum-likelihood (Mortensen) estimator whose solution is characterized by a Hamilton–Jacobi (HJ) partial differential equation. The proposed KBK filter provides a spectral, operator-theoretic realization of this nonlinear filtering problem by parameterizing the HJ value function in terms of principal Koopman eigenfunctions. This transformation converts the nonlinear estimation problem into a Riccati-type evolution in Koopman coordinates, yielding a linear-operator analogue of the classical Kalman–Bucy filter while preserving nonlinear structure in the original state variables.
We develop a path-integral formulation for computing principal Koopman eigenfunctions and introduce a dynamics-informed, characteristics-inspired basis construction for their approximation. Theoretical error bounds are derived for value-function and state-estimation approximations. Simulation results demonstrate improved performance over the extended Kalman filter and illustrate the ability of the KBK framework to operate in data-driven settings without explicit model linearization.

% This paper introduces the Kalman–Bucy- Koopman (KBK) filter, a novel framework for nonlinear state estimation derived through the spectral theory of the Koopman operator. The nonlinear estimation problem is formulated as a maximum-likelihood estimator whose solution is given in terms of solution of the Hamilton–Jacobi (HJ) partial differential equation. The KBK filter provides the first spectral, operator-theoretic realization of the Mortensen nonlinear filtering problem. By leveraging the spectral decomposition of the Koopman operator, the solution to the HJ equation is parameterized using the principal Koopman eigenfunctions, resulting in a linear representation of an otherwise nonlinear estimation problem. This spectral transformation reveals an elegant connection between nonlinear filtering and linear estimation theory, leading to a Koopman-based analogue of the classical Kalman filter. Beyond the theoretical formulation, we develop a path-integral algorithm for the numerical computation of the principal Koopman eigenfunctions required for implementing the proposed filter. We also propose a dynamics-informed basis construction method, inspired by the method of characteristics, for approximating the principal eigenfunctions.
% Simulation studies validate the theoretical developments and demonstrate the efficacy of the proposed Kalman–Bucy–Koopman filter—both in comparison with the extended Kalman filter and in data-driven estimation settings.
\end{abstract}

\section{Introduction}
The problem of nonlinear state estimation is central to modern control, robotics, signal processing, and machine learning \cite{jazwinski1970stochastic,simon2006optimal,anderson1979optimal}.
While the celebrated Kalman–Bucy filter \cite{kalman1961new} provides an optimal solution for linear systems with Gaussian noise processes, extending its analytical tractability and structural optimality to nonlinear systems remains a fundamental challenge in estimation theory.
Classical approaches such as the Extended Kalman Filter (EKF), Unscented Kalman Filter (UKF), and Particle Filters (PF) approximate the nonlinear dynamics or posterior density through local linearization, sigma-point sampling, or stochastic simulation \cite{julier2004unscented,julier2000new,doucet2001sequential}.
Despite their widespread use, these methods often lack the elegant linear–quadratic structure  that underpin the linear Kalman–Bucy framework \cite{fleming1975deterministic}.
As a result, nonlinear estimation algorithms frequently involve trade-offs between accuracy, interpretability, and computational complexity, especially in high-dimensional or data-driven settings.

The growing intersection of estimation and machine learning has further amplified the importance of nonlinear filtering.
Deep state-space models (SSMs) including the Deep Kalman Filter and its variants extend probabilistic state estimation to large-scale learning tasks \cite{krishnan2015deep,karl2017deep}. Similarly, Neural ODEs and continuous-time latent dynamics models embed estimation within differentiable architectures for sequential prediction and control 
\cite{chen2018neural}.
Recent structured state-space models such as S4 and DSS further highlight the central role of state estimation in long-range sequence modeling and modern machine learning 
\cite{gu2022s4,gu2021dss}.
Despite these advances, most learning-based estimators rely on parametric surrogates or amortized inference, with limited theoretical connection to classical filtering theory.
There remains a pressing need for estimation framework that preserve analytical structure while leveraging data-driven representations of nonlinear dynamics.

In this paper, we propose a fundamentally new approach to nonlinear state estimation based on Koopman operator theory.
The Koopman operator provides a linear, infinite-dimensional representation of nonlinear dynamics by acting on functions called observables of the system state \cite{koopman1931hamiltonian,mezic2005spectral,mezic2013analysis}.
This operator-theoretic framework enables the analysis, control, and estimation of nonlinear systems using linear operator tools \cite{mauroy2016global,mezic2021koopman,surana2016linear,williams2015data,surana2016koopman}.
Building on our prior work \cite{vaidya2025koopman} established a connection between the solution of the Hamilton–Jacobi (HJ) equation and the spectral properties specifically, the principal eigenvalues and eigenfunctions of the Koopman operator.
In the present paper, we exploit this connection in the context of nonlinear estimation.

Two popular formulations of the stochastic filtering problem, include the Zakai and Kushner–Stratonovich (KS) filters \cite{bain2009fundamentals,kushner1964differential,zakai1969optimal}.
Both formulations describe the evolution of the state-density function through the Fokker–Planck or Kolmogorov forward equation, which is the dual of the Koopman advection equation with an added diffusion term \cite{Lasota,vaidya_stochastic_linearoperator}.
Alternatively, the solution to the nonlinear filtering problem can be obtained from the HJ equation, leading to what is classically known as the maximum-likelihood or Mortensen filter \cite{mortensen1968maximum,fleming1975deterministic}.
This HJ-based formulation forms the foundation for the proposed Kalman–Bucy–Koopman (KBK) filter, which generalizes the Kalman–Bucy filter to nonlinear systems through a linear operator-theoretic formulation.

The key insight underlying our approach is that the optimal value function solving the HJ equation can be parameterized in terms of the Koopman principal eigenfunctions.
This yields a quadratic representation of the cost functional in the Koopman eigenfunction coordinates, transforming the nonlinear estimation problem into a linear estimation problem in the lifted (Koopman) space.
The resulting KBK filter evolves according to Riccati-type differential equations, analogous to the Kalman–Bucy equations, but expressed in terms of Koopman eigenfunction coordinates.
Moreover, the inclusion of the gradient and Hessian of the eigenfunctions in the gain dynamics ensures that the filter remains fully nonlinear in the original state variables while preserving linearity in the principal eigenfunction coordinates.

The literature on Koopman-based control and estimation has grown rapidly in recent years \cite{mauroy2016global,peitz2019koopman,villanueva2021towards,borggaard2009control, abraham2019active,korda2018linear, sootla2018optimal,otto2021koopman,fackeldey2020approximative,kaiser2021data,ma2019optimal,huang2018feedback,huang2022convex}.
Surana \cite{surana2016linear,surana2016koopman} proposed Koopman observer synthesis for nonlinear systems using linear models in the lifted observable space, demonstrating improved robustness compared to EKF-type estimators.
\cite{mauroy2016global,mezic2021koopman,susuki2016applied} exploited Koopman eigenfunctions for global stability, observability, and isostable analysis, highlighting their potential for globally valid coordinate transformations.
More broadly, Koopman spectral expansions and Extended Dynamic Mode Decomposition (EDMD) methods enable the data-driven identification of linear operator representations for nonlinear systems, while kernel and reproducing-kernel Hilbert space (RKHS) formulations offer convex, mesh-free frameworks for eigenfunction approximation \cite{hamzi2025kernel,lee2025kernel}.

Within this rich body of work, the proposed KBK filter provides a systematic bridge between Koopman operator theory, Hamilton–Jacobi optimal estimation, and Kalman–Bucy filtering.
Unlike previous data-driven observer approaches that rely on regression or local linearization in the lifted space, the KBK framework is derived directly from the variational structure of the estimation problem.
It offers an operator-theoretic analog of the Kalman–Bucy filter that evolves over the space of eigenfunctions yielding an estimator that is linear in the operator, yet nonlinear in the state. Unlike lift-and-linearize approaches, the KBK filter is derived from the Hamilton–Jacobi optimality principle and is not a regression-based linear approximation.

The main contributions of this paper are summarized as follows.  We establish a spectral representation of the Hamilton–Jacobi equation that arise in Mortensen nonlinear estimation in terms of Koopman principal eigenfunctions. We derive the Kalman–Bucy–Koopman (KBK) filter equations, showing that the nonlinear estimation problem can be recast as a Riccati-type evolution in the Koopman eigenfunction coordinates. Mortensen filtering problem is shown to be solved exactly under the assumption that the output mapping lies in the span of the Koopman eigenfunctions. We further provide rigorous error bounds between the true and the approximated optimal value function when the measurement function does not lies within the span of the Koopman principal eigenfunctions thereby providing approximate solution to the filtering problem.   These value-function error bounds are subsequently used to derive error bounds on the maximum-likelihood estimates. We show that the approximation error can be improved by including higher order Koopman principal eigenfunction used in the parameterization of the optimal value function. 
We present a path-integral-based approach for approximating Koopman principal eigenfunctions, enabling the practical implementation of the proposed filter. The finite-time path integral formula is also used to propose a novel dynamics–informed eigenfunction approximation method. This method relies on a set of characteristic-inspired basis functions construction that naturally encode the system dynamics.   Through simulation studies on benchmark nonlinear systems, we show that the KBK filter achieves superior performance compared with the extended Kalman filter (EKF). A key advantage of the KBK filter is that it can be implemented in a data-driven setting without explicit knowledge of the system dynamics—unlike the EKF, which requires such knowledge for linearization. We demonstrate this data-driven capability of the KBK filter through an illustrative example.
% we show using simulation studies that the KBK filter can be implemented in data-driven setting without the knowledge of true system dynamics.  

The proposed Kalman-Bucy-Koopman filter  thus introduces a new paradigm for nonlinear estimation combining the optimality and structure of the Kalman–Bucy theory, the spectral linearity of the Koopman operator, and the data-driven flexibility of modern machine learning to offer a unified foundation for operator-theoretic filtering and inference.

\section{Preliminaries and Notations}
In this section, we present some preliminaries on the spectral theory of the Koopman operator and nonlinear estimation. We refer the readers to \cite{Lasota,mezic2020spectrum,mortensen1968maximum,bain2009fundamentals} for further details on the preliminaries.

\noindent {\bf Notations}: $\mR^n$ denotes the $n$ dimensional Euclidean space. We denote by  ${\cal C}^k$ the space of $k$-times continuously differentiable functions and ${\cal C}^0$ the space of continuous function. $\mC$ denotes the set of complex numbers. 
We denote by $s_t(x)$ and $s_t(x)$ the solutions of  systems, $\dot x=f(x)$. We use the notation
\[\nabla\cdot F(x)=\sum_{i=1}^n \frac{\partial F_i}{\partial x_i},\;\;\;\; \nabla g =\left(\frac{\partial g}{\partial x_1},\ldots, \frac{\partial g}{\partial x_n}\right)\]
where $F=(F_1,\ldots, F_n)^\top$ is a vector and $g(x)$ is a scalar valued function of $x$. 
\subsection{Spectral Theory of Koopman Operator}
In this section, we provide a brief overview of existing results on the spectral theory of the Koopman operator. For more details on this topic, refer to \cite{mezic2020spectrum,mezic2021koopman}. Consider the dynamical system
\begin{align}
    \dot x={f }(x),\label{odesys}
\end{align}
defined on a state space $x\in \mR^n$.  The vector field $f$ is assumed to be smooth function. Let $\cF\subseteq \cC^0$ be the function space of observable $\psi: \mR^n\to \mC$.
% The preliminaries presented will apply to the uncontrolled dynamical system of the form $\dot \bx=\bff(\bx)$ with $n$-dimensional state space and the Hamiltonian system with $2n$-dimensional state space. 
% where $\cZ$ is an open set containing the origin.
% \begin{remark}\label{remark_generalsys}
% The main contribution of this paper is in providing two different procedures for approximating the Lagrangian submanifold and the HJ solution. These two approaches relies on the spectral properties of Koopman operator associated with uncontrolled system $\dot \bx=\bff(\bx)$ (refer to Eq. (\ref{cont_sys}))
% and the Hamiltonian system (\ref{Ham_system}). Hence, we describe the spectral theory of Koopman in the context of system (\ref{odesys}) so that the state $z$ could either correspond to state $\bx$ i.e., $\cZ=\cM$ or $z=(\bx^\top ,\bp^\top)^\top$ i.e., $\cZ=T^\star \cM$ depending upon the procedure used in the approximation.
% \end{remark} 
% \begin{remark}\label{remark_generalsys2} From Assumptions \ref{assume_system} and \ref{assumption_hamiltonianeigenvalues} it follows that the vector field $\bF$ is also $\cC^\infty$ function with the origin as the hyperbolic equilibrium point and $\bE:=\frac{\partial \bF}{\partial z}(0)$ having all eigenvalues distinct. 
% \end{remark}
We have following definitions for the Koopman operator and its spectrum. 
\begin{definition}[Koopman Operator]  The family of Koopman  operators $\mathbb{U}_t:\cF\to \cF$ corresponding to ~\eqref{odesys}   is defined as 
\begin{eqnarray}[\mathbb{U}_t \psi](x)=\psi(s_t(x)). \label{koopman_operator}
\end{eqnarray}
If in addition $\psi$ is continuously differentiable, then $\varphi(x,t):=[\mU_t \psi ](x)$ satisfies a partial differential equation \cite{Lasota} 
\begin{align}
\frac{\partial \varphi}{\partial t}=\frac{\partial \varphi}{\partial x} f=: \cK_f \varphi \label{Koopmanpde}
\end{align}
with the initial condition $\varphi(x,0)=\psi(x)$. The operator $\cK_f$ is the infinitesimal generator of $\mU_t$ i.e.,
\begin{eqnarray}
{\cal K}_{f} \psi=\lim_{t\to 0}\frac{(\mathbb{U}_t-I)\psi}{t}. \label{K_generator}
\end{eqnarray}
\end{definition}

\begin{definition}\label{definition_koopmanspectrum}[Eigenvalues and Eigenfunctions of Koopman] A function $\phi_\lambda(x)$, assumed to be at least $\cC^1$,  is said to be an eigenfunction of the Koopman operator associated with eigenvalue $\lambda$ if
\begin{eqnarray}
[\mU_t \phi_\lambda](x)=e^{\lambda t}\phi_\lambda(x)\label{eig_koopman}.
\end{eqnarray}
Using the Koopman generator, the (\ref{eig_koopman}) can be written as 
\begin{align}
    \frac{\partial \phi_\lambda}{\partial x}{ f}=\lambda \phi_\lambda\label{eig_koopmang}.
\end{align}
\end{definition}
The eigenfunctions and eigenvalues of the Koopman operator enjoy the following property \cite{mezic2020spectrum,budivsic2012applied}. 
\begin{property}\label{property1}
Let $\phi_{\lambda_1}$ and $\phi_{\lambda_2}$ are the eigenfunctions of the Koopman generator associated with eigenvalues $\lambda_1$ and $\lambda_2$ respectively. If $\phi_{\lambda_1}^{k_1}\phi_{\lambda_2}^{k_2}\in {\cal C}^1$, for $k_1,k_2\in \mR^{+}$, then it is an eigenfunctions of Koopman generator with eigenvalue $k_1\lambda_1+k_2\lambda_2$. 
\end{property}
\noindent The property is true because
\begin{align}
[\mU_t \phi_{\lambda_1}^{k_1}\phi_{\lambda_2}^{k_2}]&=[\mU_t\phi_{\lambda_1}^{k_1}\mU_t\phi_{\lambda_2}^{k_2}]=[\mU_t\phi_{\lambda_1}]^{k_1}[\mU_t\phi_{\lambda_2}]^{k_2}\nonumber\\
&=e^{(k_1\lambda_1+k_2\lambda_2)t}\phi_{\lambda_1}^{k_1}\phi_{\lambda_2}^{k_2}.
\end{align}

\begin{definition}[Koopman mode decomposition (KMD)] \label{definition_kmd} Consider a scalar-valued function $g: \mR^n\to \mR$, and assume that the function $g$ can be expanded in terms of Koopman eigenfunctions as follows.
\begin{align}
   g(x)= \sum_{\bk\in \mathbb{N}^n}^N\bar g_\bk \prod\limits_{i = 1}^{n} \phi_{i}^{k_i}(z),\label{decomposition}
\end{align}
where $\bk=(k_1,\ldots, k_n)$ and $\bar g_{\bk}$ are the Koopman modes and correspond to the projection of function $g(z)$ on the eigenfunctions, $\phi_{\lambda_1}^{k_1}(z),\ldots,\phi_{\lambda_n}^{k_n}(z)$. 
% The scalar-valued function will propagate under system dynamics as follows
% \begin{align}
% g(\bs_t(z))=[\mU_t g](z)=\sum_{\bk\in \mathbb{N}^n}^N\bar g_\bk \prod\limits_{i = 1}^{n} \phi_{\lambda_i}^{k_i}(z)e^{k_i\lambda_i t}.
% \end{align}
\end{definition}
The spectrum of the Koopman operator is, in general, infinite–dimensional and
depends sensitively on the choice of the underlying function space used for
its representation and approximation \cite{mezic2020spectrum}. Consequently,
not all spectral components are equally relevant for control or estimation
purposes.

In this work, we focus on a particular portion of the spectrum associated with
the local linearization of the nonlinear system around an equilibrium point.
Under the hyperbolicity assumption on the equilibrium of system
(\ref{odesys}), this part of the spectrum is well defined and consists of
eigenvalues that coincide with those of the Jacobian of the linearized
dynamics. We therefore restrict attention to the corresponding Koopman
eigenfunctions with eigenvalues that matches with the local linearization of the nonlinear
flow. The relevant theoretical properties summarized below follow from
\cite{mezic2020spectrum}.

While the Koopman spectrum is classically defined over the entire state space
and over infinite time horizons (cf.\ (\ref{eig_koopman})–(\ref{eig_koopmang})),
it can also be meaningfully characterized over finite time intervals or
restricted subsets of the state space. In particular, \emph{subdomain} or
\emph{open} eigenfunctions are defined  on subsets of the state space
and provide a well-posed spectral description of the dynamics in a
neighborhood of the equilibrium \cite[Definition~5.2, Lemma~5.1]{mezic2020spectrum}.

The filter construction in this paper employs the \emph{principal
eigenfunctions} of the Koopman operator. These eigenfunctions are locally
defined and possess eigenvalues that match those of the linearization of the
nonlinear system at the equilibrium \cite[Proposition~5.8]{mezic2020spectrum}.
They therefore furnish intrinsic coordinates that linearize the dynamics to
first order and naturally connect to classical results such as the
Hartman--Grobman theorem and Poincaré normal forms \cite{arnold2012geometrical}.

Depending on the regularity of the flow, these principal eigenfunctions may be
defined either on a proper subset $\mathcal P \subset \mathbb{R}^n$
(subdomain eigenfunctions) or globally on the entire state space
\cite[Lemma~5.1, Corollaries~5.1--5.2, 5.8]{mezic2020spectrum}. In the sequel,
these principal eigenfunctions serve as the lifting coordinates underlying the
proposed spectral filter dynamics.

\subsection{Estimation in Nonlinear Systems}\label{section_nonlinearestimation}
For the nonlinear estimation problem, we consider system with output measurements of the form
% Let the state \(x(t) \in \mathbb{R}^n\) evolve as an Itô stochastic differential equation
\begin{align}
    \dot x = f(x) + w(t), \;\;y=h(x)+v(t)\label{system_measurement}
\end{align}
where $x\in \mR^n$ and $y\in \mR^p$ are the state and output measurement respectively. $w(t)\in \mR^n$  and  $v(t)\in \mR^p$ are a white Gaussian noise processes assumed to be independent with covariance matrices $R\in \mR^{n\times n}$ and $Q\in \mR^{p\times p}$ respectively i.e.,
\[E[w(t)]=0,\;\;\;E(w(t)w^\top(\tau))=R\delta(t-\tau)\]
\[E[v(t)]=0,\;\;\;E(v(t)v^\top(\tau))=Q\delta(t-\tau).\]
The stochastic differential equation in (\ref{system_measurement}) is assumed to be in the sense of Itô.  Both the vector field $f$ and the output mapping $h$ are assumed to be smooth function of $x$. For the estimation problem, the control input is typically ignored as it is assumed to known. 
The initial prior probability  density of the state $x$ is denoted by
\begin{equation}
    p(x,0) = p_0(x), \qquad \int_{\mathbb{R}^n} p_0(x)\,dx = 1.
\end{equation}

% \subsection{Prediction: The Fokker--Planck Equation}

Between measurements, the state density \(p(x,t)\) evolves deterministically according to the
\textit{Fokker--Planck (FP)} or \textit{Kolmogorov forward} equation:
\begin{align}
    \frac{\partial p}{\partial t}(x,t)
    = -\nabla\!\cdot\!\big(f(x)\,p(x,t)\big)\nonumber\\
    + \frac{1}{2}\sum_{i,j=1}^n
        \frac{\partial^2}{\partial x_i \partial x_j}
        \Big(R_{ij}\,p(x,t)\Big).
\end{align}
The above equation can be written in the compact form using the operator notation as follows.  
\begin{equation}
    \frac{\partial p}{\partial t} = \mathcal{L}^* p,
    \qquad \mathcal{L}^* p := -\nabla\!\cdot(f p)
    + \tfrac{1}{2} \nabla\!\cdot\!\big(R\nabla p\big).
\end{equation}
The unnormalized density 
$\bar{p}(x,t)$ satisfies the \textit{Zakai equation}:
\begin{equation}
    \frac{\partial \bar{p}}{\partial t} = \mathcal{L}^* \bar{p}
    + \bar{p}\,h(x)^\top Q^{-1}y.
\end{equation}
 This is a linear Stochastic partial differential equation (SPDE), linear in $\bar p$, but stochastic because $y$ is random. The normalized posterior \(p =\frac{ \bar{p} }{ \int \bar{p} dx}\)
satisfies the \textit{Kushner--Stratonovich (KS)} equation:
\begin{align}
    \frac{\partial p}{\partial t}&= \mathcal{L}^* p
    + p\,\big(h - \bar{h}\big)^\top Q^{-1}\,\big(y - \bar{h}\big),\nonumber\\  
    \bar{h}(t) &= \int h(x)\,p(x,t)\,dx.
\end{align}

Thus, the prediction step is governed by the FP operator \(\mathcal{L}^*\),
and the correction step introduces information from the measurements. The Kushner–Stratonovich equation is a nonlinear SPDE because normalization introduces division by a random integral $\int \bar p dx$. 
%\subsection{Estimators from the Posterior Density}
There are different ways to build an estimator given the posterior \(p(\cdot,t)\). In particular, we can use\\
    \textit{MMSE estimator:} 
    \[\hat{x}_{\mathrm{MMSE}}(t) = \int x\,p(x,t)\,dx. \]
 \textit{MAP estimator:} 
    \[ \hat{x}_{\mathrm{MAP}}(t) = \arg\max_x p(x,t). \]
   \textit{Credible sets:} 
    regions \[\mathcal{C}_\alpha(t) {\rm s.t.} 
    \int_{\mathcal{C}_\alpha} p(x,t)\,dx = \alpha.\]

% Smoothing (estimating \(x(\tau)\) for \(\tau < t\)) uses a forward FP/KS propagation and a backward adjoint pass.

\subsection{Special Case: Linear--Gaussian (Kalman--Bucy Filter)}
For the special case when the system dynamics and the measurement model are linear, we obtain the Kalman-Bucy filter. Consider the following linear time invariant system 
\begin{align}
    \dot x &= A x + w(t),\\
    y &= Cx + v(t),
\end{align}
with $w(t)$ and $v(t)$ are independent Gaussian white noise with covariance matrices $R$ and $Q$ respectively. The FP solution remains Gaussian with mean \(\hat x(t)\) and covariance \(P(t)\)
satisfying the following \textit{Kalman--Bucy equations}:
\begin{align}
    \dot{\hat{x}} &= A \hat{x}+ K(t)\,\big(y - C\hat{x}\big),
    \qquad K(t) = S(t) C^\top Q^{-1}, \\[4pt]
    \dot{S} &= A S + S A^\top + R - S C^\top Q^{-1} C S.
\end{align}
with initial covariance $S(0)=S_0$.  
This linear Gaussian filter is a finite-dimensional projection of the FP/KS filter.

\section{ Kalman-Bucy-Koopman Filter via HJ Formulation of Optimal Estimation}

Instead of evolving the full probability density (as in the Fokker--Planck or Kushner--Stratonovich equations), we may seek the most probable trajectory \(x(t)\) given the observations \(y(t)\).  
This leads to the maximum a posteriori (MAP) or minimum-energy estimation problem also called as Mortensen filter. For the MAP or minimum-energy estimation problem we  consider nonlinear stochastic system with output measurement where we borrow the set-up from \ref{section_nonlinearestimation}
\begin{align}
    \dot{x} = f(x) + \,w(t),\;
    \qquad y = h(x) + v(t),\label{sys_output}
\end{align}
with \(w(t)\) and \(v(t)\) are  assumed to be independent Gaussian white noises with covariances matrices $R$ and $Q$ respectively. 
The goal is to estimate the hidden state \(x(t)\) from the noisy observations \(y(t)\). The initial state is $x(0)$, where $x(0)$ is a random vector with Gaussian distribution with covariance matrix  $\Sigma$ and mean $\mu$. The cost functional for the maximum posteriori problem is given by
\begin{align}
    J_t[x(\cdot)] =\frac{1}{2}(x-\mu)^\top \Sigma^{-1}(x-\mu)\nonumber\\+ \frac{1}{2}\int_{0}^t 
    \Big[
        (\dot{x} - f(x))^\top R^{-1} (\dot{x} - f(x))\nonumber\\
        + (y - h(x))^\top Q^{-1} (y - h(x))
    \Big]\,d\tau.
    \label{eq:action}
\end{align}
\begin{align}
    J_t[x(\cdot)] =V_0(x)+ \frac{1}{2}\int_{0}^t 
    \Big[
        (\dot{x} - f(x))^\top R^{-1} (\dot{x} - f(x))\nonumber\\
        + (y - h(x))^\top Q^{-1} (y - h(x))
    \Big]\,d\tau.
    \label{eq:action}
\end{align}
This cost functional penalizes the deviations from the model dynamics (\(w\)-term)  and the
 mismatch between predicted and observed measurements (\(v\)-term).

For the computation of the maximum-likelihood estimate  the output $y(\tau)$, $0\leq \tau\leq t$, over a time interval.  Substituting for $y(\tau)$ into $J_t$ in (\ref{eq:action}), then it is functional only of $x(\tau)$, $0\leq \tau\leq t$ with both end conditions $x(0)$ and 
$x(t)$ free. Using standard variational methods, we determine the trajectory
$x^\star(t)$, $0\leq \tau\leq t$,  which minimizes $J_t$. The modal-trajectory estimate is
then $x^\star(t)$, the final state of the most probable trajectory at time $t$, given
$y(\tau)$, $0\leq \tau\leq t$. Only the final value of the
minimizing trajectory is of any significance. We refer to \cite{mortensen1968maximum} for more detailed discussion on this.

Suppose now that we wish to construct a sequential estimator. This
means that, as $t$ varies, we wish continually to update the optimal estimate.
Let $x^\star_{[0,t]}(\tau), 0\leq \tau\leq t$, denote the trajectory which minimizes $J_t$ when $y(\tau),
0 \leq \tau\leq t$ is given. Let $\hat x(t)$ denote the optimal maximum-likelihood estimate
at time $t$ when $y(\tau), 0\leq \tau\leq t$ is given.

The Hamilton Jacobi theory can be applied to solve this problem. We first define

\begin{align}
V(x,t)=\min_{u(\tau),0\leq \tau\leq t}J_t
\end{align}
It is known that the $V(x,t)$ is the viscosity solution of the following HJ equation \cite{krener2004convergence}
% \begin{align}
% \frac{\partial V}{\partial t}+H(x,\nabla V_x,t)=0
% \end{align}
\begin{align}
\frac{\partial V}{\partial t}+\frac{\partial V}{\partial x} f(x)
    +\frac{1}{2} \frac{\partial V}{\partial x} R\frac{\partial V}{\partial x}^\top \nonumber\\
    -\frac{1}{2}(y(t) - h(x))^\top Q^{-1} (y(t) - h(x)).
\label{HJequation}
\end{align}
 and the initial condition for the PDE given by
% \begin{align}V(x,0)=\frac{1}{2}(x-\mu)^\top \Sigma^{-1} (x-\mu)\label{boundaryconditionHJ}
% \end{align}
\begin{align}V(x,0)=V_0(x)
\label{boundaryconditionHJ}
\end{align}
with the minimum of this function at $x=\mu$. We make the following assumption on the solution of the HJ PDE.
\begin{assumption}\label{assumption_HJviscositysolution} The Hamiltonian corresponding to the HJ PDE (\ref{HJequation}) is quadratic in the gradient $\frac{\partial V}{\partial x}$ and hence convex. We assume that the HJ PDE (\ref{HJequation}) admits a viscosity solution which is  semi-continuous \cite{CrandallIshiiLions1992}.  
\end{assumption}
The optimal estimate $\hat x(t)$ is obtained as the root of following equation
\begin{align}\nabla V(x,t)=0\label{optimalcondition}
\end{align}
The optimal or maximal likelihood estimate, $\hat x(t)$, is the root of the above equation. Computationally and from the online implementation perspective it is more efficient to obtain the differential equation governing  the evolution of $\hat x(t)$. This governing equation is derived as follows \cite{mortensen1968maximum}. Taking time derivative of (\ref{optimalcondition}), we obtain
\begin{align}
\frac{d}{dt}\left(\frac{\partial V(x,t)}{\partial x}\right)\Big|_{x=\hat x(t)}=\frac{\partial^2 V }{\partial x \partial  t}+\frac{\partial^2 V}{\partial x^2}\frac{d\hat x}{dt}\Big|_{x=\hat x(t)}
\end{align}
Let, 
\begin{align}\Pi_{ij}(x,t):=\frac{\partial ^2 V}{\partial x_i x_j}(x,t)\label{Hessian}
\end{align}
Using the Hessian notation and (\ref{HJequation}), we can write the above as 
\begin{align}
\frac{d}{dt}V_x(x,t)\Big|_{x=\hat x(t)}=\Big\{-H_x(x,\nabla_x V(x,t),t)\nonumber\\
+\Pi(x,t)\frac{d\hat x}{dt}\Big\}_{x=\hat x(t)}
\end{align}

Under the assumption that the Hessian of $V(x,t)$ i.e, the matrix $\Pi$ in  (\ref{Hessian}) is invertible it can be shown that the the differential equation for the evolution of the maximum likelihood estimate is given by \cite{mortensen1968maximum}
\begin{align}
\dot {\hat x}=f(\hat x)+\Pi^{-1}(\hat x,t)\frac{\partial h}{\partial x}(\hat x)Q^{-1}(y-h(\hat x))\label{diffeqn_maximalestimate}
\end{align}
The initial condition for the above differential equation is $\hat x(0)=\mu$. We next proceed with deriving the maximum likelihood filter equations in the Koopman eigenfunction coordinates.
We make following assumption on the system dynamics for the estimator design problem using Koopman.

\begin{assumption}\label{assume_peig} We assume that the origin is an hyperbolic equilibrium point of the system $\dot x=f(x)$ and $\Phi(x)\in \mR^n$ are the principal eigenfunctions of the Koopman operator associated with the matrix of eigenvalue $\Lambda$ i.e.,
\begin{align}
\frac{\partial \Phi}{\partial x}f(x)=\Lambda \Phi(x),\label{Koopmaneigfunction}
\end{align}
The matrix $\Lambda$ is in real Jordan canonical form as the principal eigenfunctions are assumed to be real. We assume that the eigenfunctions are well defined in  $\mR^n$.  
\end{assumption}
\begin{remark}
The principal eigenfunctions are assumed to be well defined on $\mathbb{R}^n$ in the above. 
In general, however, principal eigenfunctions  may only be defined on an invariant subdomain $\mathcal{P} \subseteq \mathbb{R}^n$. 
In such cases, the validity of the proposed proposed filter is restricted to the region where the eigenfunctions are well defined. 
Extension of the KBK filter beyond a single invariant subdomain, for example via patching of local eigenfunctions, will be an interesting direction for future work.
\end{remark}
Following (\ref{eq:action}), we choose the following cost functional for the maximum posterior problem in the spectral Koopman coordinates 
\begin{align}
    &\cJ_t[u(\cdot)] =\frac{1}{2}(\Phi(x)-\Phi(\mu))^\top \Sigma^{-1}(\Phi(x)-\Phi(\mu))\nonumber\\&+ \frac{1}{2}\int_{0}^t 
    \Big[
        u(\tau)^\top R^{-1}_1(x) u(\tau)
        + (y - h(x))^\top Q^{-1} (y - h(x))
    \Big]\,d\tau.
    \label{eq:action_koopman}
\end{align}
with $0\leq t\leq T$ and where $R_1(x):=\frac{\partial \Phi}{\partial x}^\top R^{-1} \frac{\partial \Phi}{\partial x}$. The objective is to choose $x(t)$ to minimize $\cJ_t$ with $x(0)$ free subject to the system equation constraints (\ref{sys_output}). Comparing the objective function (\ref{eq:action_koopman}) for the proposed Kalman-Bucy-Koopman filter with the objective function (\ref{eq:action}), we notice the  presence of the state dependent term, $R_1(x)$ in the cost function. 
Importantly, the introduction of the state-dependent weighting $R_1(x)$ does not alter the underlying stochastic model; rather, it arises naturally from expressing the Mortensen minimum-energy functional under the nonlinear coordinate transformation induced by the Koopman eigenfunctions. In this sense, the proposed formulation is equivalent to the classical Mortensen problem in the original state variables, with the Jacobian of the Koopman transformation accounting for the change of metric in the process noise penalty.
% The use of $R_1(x)$
% ensures that the HJ equation associated with the action functional becomes linear in Koopman eigenfunction coordinates, enabling an explicit quadratic parameterization of the value function. Additionally this formulation preserves the structure of the classical Mortensen problem under nonlinear coordinate changes induced by Koopman eigenfunctions.
% The state dependent term, $R_1(x)$, is introduced to facilitate the analytical solution of the associated HJ equation in the Koopman eigenfunction coordinates.
The initial condition is assumed to be quadratic but in Koopman eigenfunctions coordinates $\Phi$ with minimum at $x=\mu$, which is the estimate of the initial state $x(0)$. It is also important to emphasize that the term $\frac{\partial \Phi}{\partial x}$ in $R_1(x)$ is an invertible matrix as $\Phi$ forms a change of coordinates transforming the nonlinear system to a linear system (Assumption \ref{assume_peig}).
% The Hamiltonian associated with the optimal control problem can be written as 
% \begin{align}
%     \cH(x,p,t)
%     = p^\top (f(x)+u)
%     +\frac{1}{2} p^\top R_1^{-1}(x)\,p\nonumber\\
%     +\frac{1}{2}(y(t) - h(x))^\top Q_1^{-1} (y(t) - h(x)).
%     \label{eq:hamiltonian}
% \end{align}
% Applying optimality principle, we obtain
% \[\frac{\partial \cH}{\partial u}=0\implies u=-R_1(x)p\]
% Substuiting the optimal value of $u$ in the Hamiltonian we obtain
% \begin{align}\cH(x,p,t)=p^\top f-\frac{1}{2}p^\top R_1^{-1}(x)p\nonumber\\+\frac{1}{2}(y(t) - h(x))^\top Q_1^{-1} (y(t) - h(x))\label{hamiltonian}
% \end{align}
The optimal value function, satisfies 
\begin{align}
\cV(x,t)=\min_{u(\tau),0\leq \tau\leq t}\cJ_t\label{optimalvaluefunction_Koopman}
\end{align}
The optimal value function, $\cV(x,t)$, satisfies following HJ PDE.
\begin{align}
\partial_t \cV+H(x,\nabla_x \cV(x,t),t)=0
\label{HJequationPhi}
\end{align}
where $H$ is the Hamiltonian and of the form
\begin{align}H(x,\nabla_x \cV(x,t),t)=\frac{\partial \cV}{\partial x} f(x)
    +\frac{1}{2} \frac{\partial \cV}{\partial x} R_1^{-1}(x)\frac{\partial \cV}{\partial x}^\top \nonumber\\
    -\frac{1}{2}(y(t) - h(x))^\top Q^{-1} (y(t) - h(x)).
    \end{align}
with the initial condition given by 
\begin{align}\cV(x,0)=\frac{1}{2}\left(\Phi(x)-\Phi(\mu)\right)^\top \Sigma^{-1}\left(\Phi(x)-\Phi(\mu)\right)\label{boundary_Koopmancoordinates}
\end{align}

This section establishes quadratic parameterizations of the optimal value function using the Koopman principal eigenfunctions $\Phi$
and their higher-order products 
$\bar \Phi(x)$, leading to the representations $\cV_\Phi(x,t)$ 
 and $\cV_{\bar \Phi}(x,t)$. The main results are summarized as follows.

First, Theorem~\ref{theorem1} provides an exact representation of the optimal value function in terms of the principal eigenfunctions under the assumption that the output mapping 
$h(x)$ lies in the span of $\Phi(x)$ (Assumption~\ref{assume_span}). Under this condition, the solution of the Hamilton–Jacobi equation~\eqref{HJequationPhi} admits a quadratic form in Koopman eigenfunction coordinates, yielding $\cV(x,t)=\cV_\Phi(x,t)$
  as given in~\eqref{optimalcostKoopmanparamter}. Hence, under Assumption \ref{assume_span}, the Mortensen filter problem is solved exactly. 

Second, Theorem~\ref{theorem2} relaxes Assumption~\ref{assume_span} to provide approximate solution to the Mortensen filter problem. Theorem~\ref{theorem2} characterizes the resulting approximation error when 
$\cV_\Phi(x,t)$ is used in place of the true value function 
$\cV(x,t)$. Explicit bounds are derived on the value-function approximation error, which are subsequently used to quantify the discrepancy between the true maximum-likelihood state estimates and those obtained from the principal-eigenfunction-based approximation.

Finally, Theorem~\ref{thm:lifted_bounds}, presented in Subsection~\ref{section_KMD}, extends the parameterization framework to include higher-order products of the Koopman eigenfunctions. This enriched representation mitigates approximation error when the output mapping $h(x)$
 cannot be adequately captured by the principal eigenfunctions alone and yields computable upper and lower bounds on the optimal value function $\cV(x,t)$.

For the first main result, we  make the following assumption.
\begin{assumption}\label{assume_span}
We assume that the output function $h(x)$ lies in the span of the Koopman principal eigenfunction i.e., $h(x)\in {\rm span}\{\phi_1,\ldots,\phi_n\}$.
\end{assumption}
Following Assumption \ref{assume_span}, we know that there exists a matrix $C\in \mR^{p\times n}$ such that 
\[h_i(x)=c_i^\top \Phi,\;\;\;i=1,\ldots, p\]
where, $c_i^\top $, forms the rows of matrix $C$.   
The $c_i's$ are 
obtained from the solution of the following least square optimization problem. 
\begin{align}
\min_{c_i} \left<h_i-c_i^\top \Phi,h_i-c_i^\top \Phi\right>,\;\;i=1,\ldots,n\label{optimization_prob}
\end{align}
% {\color{red} 
% \[<h,h>-2<h,C\Phi>+<C\Phi,C\Phi>\]
% }
where $\left<\cdot,\cdot\right>$ stands for inner product in the function space
\[\left<f,g\right>:=\int_{\mR^n} f(x)g(x)dx\]
It is known that the solution to this optimization problem (\ref{optimization_prob}) is given by 
\begin{align}\mR^n\ni c_i=G^{-1}b_i,\;\;\;i=1,\ldots,n.\label{step2}
\end{align}
where 
\begin{align} G=\left<\Phi,\Phi^\top\right>\in \mR^{n\times n}, \;\;\; b_i=\left<\Phi,h_i\right>\in \mR^{n}\label{Gb_definition}
\end{align}
The numerically expensive step of computing the integral involved in the computation of matrix $C$ can be replaced with the empirical sum over some sample data points \cite{vaidya2025koopman}. In particular, let $\{x_k\}_{k=1}^L$ be the  data points sampled uniformly. We can define 
\begin{align}
\hat G=\frac{1}{L}\sum_k^L \Phi(x_k)\Phi(x_k)^\top,\;\;\;\hat b_i=\frac{1}{L}\sum_k^L \Phi(x_k)h_i(x_k)
\end{align}
The $\hat c_i$ is then obtained as 
\[\hat c_i=\hat G^\dagger \hat b_i\]
where $\dagger$ stands for pseudo inverse. 
% It is easy to show that with the above choice of $C$, we have 
% $h(x)-C\Phi(x)=0$ i.e., the residual is zero following Assumption \ref{assume_span}. We have Theorem \ref{theorem2}, where we  will relax this Assumption. 

The following is the first main result of this paper on the realization of Kalman-Bucy filter in the spectral Koopman coordinates. 
\begin{theorem}\label{theorem1}
Under Assumption \ref{assume_span}, the optimal value function (\ref{optimalvaluefunction_Koopman}) for the solution of HJ equation (\ref{HJequationPhi}) is given by the following expression
\begin{align}
\cV_\Phi(x,t)=\frac{1}{2}\Phi^\top(x)P(t)\Phi(x)+s^\top(t)\Phi(x)+r(t)\label{optimalcostKoopmanparamter}
\end{align}
where $P(t)\in \mR^{n\times n}$ is a positive matrix, $s(t)\in \mR^n$, and $r(t)\in \mR$ satisfy
\begin{align}&\dot P(t)+\Lambda^\top P(t)+P(t)\Lambda+P(t)R P(t)-C^\top Q^{-1}C=0\nonumber\\
&\dot s(t) + \Lambda^\top s(t) +P(t) R s(t)
+C^\top Q^{-1} y(t)=0\nonumber\\
&\dot r(t)+\frac{1}{2}s^\top(t) R s(t)-\frac{1}{2}y^\top(t) Q^{-1}y(t)=0\label{estimator}
\end{align}
for $t\in [0,T]$ and with initial condition 
\begin{align}
 P(0)=\Sigma^{-1},\;s(0)=-\Sigma^{-1}\Phi(\mu),\;r(0)=\frac{1}{2}\Phi(\mu)^\top \Sigma^{-1}\Phi(\mu)\label{boundarycondition_theorem}
\end{align}
\end{theorem}
\begin{proof}
For the case when the Assumption \ref{assume_span} is satisfied, we write $h(x)=C\Phi(x)$, where the matrix $C$ is obtained from the solution of the optimization problem (\ref{optimization_prob}). 
Substituting the parameterization of the optimal value function in Eq. (\ref{optimalcostKoopmanparamter}) in the HJ equation (\ref{HJequationPhi}), and using the Assumption \ref{assume_span} we obtain

\begin{align}
\left(\frac{1}{2}\Phi^\top \dot P \Phi+\dot s^\top \Phi+\dot r\right)+\Phi^\top P\frac{\partial \Phi}{\partial x}f(x)+s^\top \frac{\partial \Phi}{\partial x}f(x)\\+\frac{1}{2} \Phi^\top P R P\Phi +\Phi^\top PR s+\frac{1}{2}s^\top R s \\
-\frac{1}{2}(y-C\Phi(x))Q^{-1}(y-C\Phi)=0
\end{align}
Using (\ref{Koopmaneigfunction}) and collecting terms involving quadratic, linear and no terms involving $\Phi$, we obtain
\begin{align}
&\dot P+\Lambda^\top P+P\Lambda+PRP-C^\top Q^{-1}C=0\nonumber\\
&\dot s +\Lambda^\top s  +PRs+C^\top Q^{-1}y=0\nonumber\\
&\dot r+\frac{1}{2}s^\top Rs-\frac{1}{2}y^\top Q^{-1}y=0
\end{align}
The desired boundary conditions (\ref{boundarycondition_theorem}) is obtained by equating  (\ref{optimalcostKoopmanparamter}) with (\ref{boundary_Koopmancoordinates}) at $t=0$.
% to obtain
% % \[\frac{1}{2}\Phi^\top \Sigma^{-1}\Phi=\frac{1}{2}\Phi^\top P(0)\Phi\]
% % \[-\Phi^\top(\mu)\Sigma^{-1}\Phi(x)=s^\top(0)\Phi\]
% % \[r(0)=\frac{1}{2}\Phi(\mu)^\top \Sigma^{-1}\Phi(\mu)\]

% \[ P(0)=\Sigma^{-1},\;s(0)=-\Sigma^{-1}\Phi(\mu),\;r(0)=\frac{1}{2}\Phi(\mu)^\top \Sigma^{-1}\Phi(\mu)\]
\qed
\end{proof}
Let $\Pi_\Phi(x,t)$ be the Hessian of the optimal value function $\cV_\Phi$ from (\ref{optimalcostKoopmanparamter}) with entries
\begin{align}
[\Pi_\Phi]_{ij}=\frac{\partial^2 }{\partial x_i\partial x_j}\left(\frac{1}{2}\Phi^\top(x)P(t)\Phi(x)+s^\top(t)\Phi(x)+r(t)\right)
\end{align}
The Hessian matrix is explicitly given by

\[
\Pi_\Phi
= \frac{\partial \Phi}{\partial x}^{\top} P(t)\, \frac{\partial \Phi}{\partial x}
+ \sum_{i=1}^{n}\big(P(t)\,\Phi(x)+s(t)\big)_i\, \frac{\partial^2 \phi_i(x)}{\partial x^2},
\]

% {\color{red}{
% \textbf{Gradient:}
% \[
% f(x)=\frac{1}{2}\,\Phi(x)^{\top} P \,\Phi(x)+\Phi(x)^{\top} s,\quad
% J(x):=\frac{\partial \Phi(x)}{\partial x}\in\mathbb{R}^{n\times n},\]\[
% \nabla f(x)=J(x)^{\top}\big(P\,\Phi(x)+s\big).
% \]

% \textbf{Hessian:}
% \[
% \nabla^2 f(x)
% = J(x)^{\top} P\, J(x)
% + \sum_{i=1}^{n}\big(P\,\Phi(x)+s\big)_i\, \nabla^2 \Phi_i(x),
% \]
% where \(\nabla^2 \Phi_i(x)\in\mathbb{R}^{n\times n}\) is the Hessian of the \(i\)-th component of \(\Phi\).

% \textbf{Linear special case} (\(\Phi(x)=A x\), \(A\in\mathbb{R}^{n\times n}\)):
% \[
% J(x)=A,\ \ \nabla^2 \Phi_i(x)=0
% \;\Rightarrow\;
% \nabla^2 f(x)=A^{\top} P A \ \text{(constant)}.
% \]}

% \[\]
Combining the dynamics of the maximum likelihood estimates of the state (\ref{diffeqn_maximalestimate}) with the (\ref{estimator}), we obtain the following  dynamics for the {\it Kalman-Bucy-Koopman} filter.

\[
\boxed{
\begin{aligned}
\dot {\hat x}&=f(\hat x)+\Pi^{-1}_{\Phi}(\hat x,t)\frac{\partial h}{\partial x}(\hat x)Q^{-1}(y-h(\hat x)) \\
0&=\dot P+\Lambda^\top P+P\Lambda+PRP-C^\top Q^{-1}C 
 \\0&=\dot s+\Lambda^\top s  +PRs+C^\top Q^{-1}y\\
0&=\dot r+\frac{1}{2}s^\top Rs-\frac{1}{2}y^\top Q^{-1}y\\
P(0)&=\Sigma^{-1},\;s(0)=-\Sigma^{-1}\Phi(\mu)\\r(0)&=\frac{1}{2}\Phi(\mu)^\top \Sigma^{-1}\Phi(\mu)\\
\Pi_\Phi
&= \frac{\partial \Phi}{\partial x}^{\top} P(t)\, \frac{\partial \Phi}{\partial x}
+ \sum_{i=1}^{n}\big(P(t)\,\Phi(x)+s(t)\big)_i\, \frac{\partial^2 \phi_i(x)}{\partial x^2}
\end{aligned}
}
\]

% Note that for the case when Assumption \ref{assume_span} is satisfied we have $\cV_\Phi$ from Eq. (\ref{optimalcostKoopmanparamter}) equal to $\cV$ from Eq. (\ref{optimalvaluefunction_Koopman}) i.e., $\cV_\Phi=\cV$. 
% For the case where the Assumption \ref{assume_span} is not satisfied we have $\cV_\Phi\neq \cV$ and the objective is determine the error between $\cV$ and $\cV_\Phi$. 
We now relax the Assumption \ref{assume_span}. In particular, when the Assumption \ref{assume_span} is not satisfied there will be the error between the projection $h(x)$ on the space span by the eigenfunctions $\Phi$ i.e., $C\Phi$. Let $\Delta (x)$ be the projection error i.e., 
\[\Delta(x)=h(x)-C\Phi.\]
where $C$ is obtained from the solution of the optimization problem (\ref{optimization_prob}). We make following assumption on the projection error. 
\begin{assumption}\label{assumption_approximation}
We assume that 
\begin{align}
\delta_h&=\sup_{x\in \mR^n} \|h(x)-C\Phi(x)\|\\
\delta(t)&=\sup_{x\in  \mR^n}\|y(t)-h(x(t))\|,\;\;\;t\in[0,T]
\end{align}
where $C$ is obtained as the solution of the optimization problem solved in (\ref{optimization_prob}). Furthermore, the optimal value function $\cV$ and $\cV_\Phi$ are locally strongly convex in the neighborhood of the minimizer i.e., there exists a positive $\gamma$ and $\gamma_\Phi$ such that
\begin{align}
\cV(x,t)&\geq \cV(\hat x(t),t)+\frac{\gamma}{2}\|\hat x(t)-x\|,\;\;\;\forall x\nonumber \\
\cV_\Phi(x,t)&\geq \hat \cV(\hat x_\Phi(t),t)+\frac{\gamma_\Phi}{2}\|\hat x_\Phi(t)-x\|,\;\;\;\forall x
\end{align}
\end{assumption}
 
Following results can be proved when the Assumption \ref{assume_span} is not satisfied i.e., when the output function, $h(x)$, does not belong to the span of Koopman principal eigenfunctions.
\begin{theorem}\label{theorem2} Let the Assumption \ref{assume_span} is not satisfied but the Assumption \ref{assumption_approximation} holds true. We have following error bounds between the solution of the optimal function, $\cV$, from Eq.  (\ref{optimalvaluefunction_Koopman}) obtained as the solution of the HJ equation (\ref{HJequationPhi}) and the one obtained based on the parameterization of Koopman eigenfunctions  i.e.(\ref{optimalcostKoopmanparamter}).

\begin{align}
|\cV-\cV_\Phi|\leq \lambda_{max}(Q^{-1})\int_{0}^t \left(\frac{\delta_h}{2}+\delta(\tau)\right)d\tau =:c(t)
\end{align}
for $t\in[0,T]$ and 
\begin{align}
\|\hat x(t)-\hat x_\Phi(t)\|\leq 2\sqrt{\frac{c(t)}{\gamma_{min}}}
\end{align}
where $\gamma_{min}=\min\{\gamma,\gamma_\Phi\}$. 
\end{theorem}

\begin{proof}
We know that the optimal value function $\cV$ and $\cV_\Phi$ (\ref{optimalvaluefunction_Koopman}) satisfy following PDEs
\begin{align}\frac{\partial \cV}{\partial t}+\frac{\partial \cV}{\partial x}f(x)+\frac{1}{2}\frac{\partial \cV}{\partial x}R_1^{-1}(x)\frac{\partial \cV}{\partial x}^\top\nonumber\\ -\frac{1}{2}(y(t)-h(x))Q^{-1}(y(t)-h(x))
\end{align}

\begin{align}\frac{\partial \cV_\Phi}{\partial t}+\frac{\partial \cV_\Phi}{\partial x}f(x)+\frac{1}{2}\frac{\partial \cV_\Phi}{\partial x}R_1^{-1}(x)\frac{\partial \cV_\Phi}{\partial x}^\top \nonumber\\-\frac{1}{2}(y(t)-C\Phi) Q^{-1}(y(t)-C\Phi(x))
\end{align}

% \[\frac{\partial \cV}{\partial t}+\cH(x,\nabla_x \cV,t),\;\;\frac{\partial \cV_\Phi}{\partial t}+\cH_\Phi(x,\nabla_x \cV_\Phi,t)\]
% where the Hamiltonian $\cH$ is as given in (\ref{hamiltonian}) and the $\cH_\Phi$ is defined as follows
% \begin{align}\cH_\Phi=p^\top f+\frac{1}{2}p^\top R_1^{-1}(x)p\nonumber\\-\frac{1}{2}(y(t) - C\Phi(x))^\top Q^{-1} (y(t) - C\Phi(x))
% \end{align}
Let 
\[r(x)=h(x)-C\Phi(x)\implies C\Phi=h-r\]
We have

\begin{align}
\cH-\cH_\Phi&=-\frac{1}{2}(y-h)^\top Q^{-1}(y-h)\nonumber\\&+(y-C\Phi)^\top Q^{-1}(y-C\Phi)
\end{align}

% \begin{align}\cH-\cH_\Phi=-\frac{1}{2}\left((y-h)^\top Q^{-1}(y-h)\\+(y-C\Phi)^\top Q^{-1}(y-C\Phi)\right)
%  \end{align}

% \[\frac{1}{2}\left((y-h)^\top Q^{-1}(y-h)-(y-h+r)^\top Q^{-1}(y-h+r)\right)\]

\[|\cH-\cH_\Phi|\leq |(y-h)^\top Q^{-1}r|+\frac{1}{2}rQ^{-1}r\]

Following Assumption \ref{assumption_approximation}, this is equal to 

\[|\cH-\cH_\Phi|\leq \lambda_{max}(Q^{-1})\delta_h\left(\frac{\delta_h}{2}+\delta(t)\right) \]

Hence 
\[|\cV-\cV_\Phi|=\left|\frac{\partial }{\partial t}(\cH_\Phi-\cH)\right|\leq  \lambda_{max}(Q^{-1})\int_{0}^t \left(\frac{\delta_h}{2}+\delta(\tau)\right)d\tau \]

It is known that the optimal maximum likelihood estimate, $\hat x(t)$, is obtained from the boundary condition $p(t)=0$ i.e., where the gradient of optimal value function $\nabla_x\cV(x,t)=0$ as $p(t)=\nabla_x \cV(x,t)$ (\ref{optimalcondition}). Hence, the error bounds between the true maximum likelihood estimate obtained from $\cV$.  and the approximate one  from $\cV_\Phi$ can be obtained as follows. 
\begin{align}
\cV(\hat x_\Phi(t),t) \leq \cV_\Phi(\hat x_\Phi(t),t)+c(t) \leq \cV_\Phi(\hat x(t),t)+c(t)\nonumber\\ \leq \cV(\hat x(t),t)+2c(t).
\end{align}

In the above we have use the fact that  $|\cV-\cV_\Phi|\leq c(t)$ and $\hat x_\Phi$ is the minimizer of $\cV_\Phi$. 
Similarly, we have inequality in the other direction as
\begin{align}
\cV_\Phi(\hat x(t),t) \leq \cV(\hat x(t),t)+c(t) \leq \cV(\hat x_\Phi(t),t)+c(t)\nonumber\\ \leq \cV_\Phi(\hat x_\Phi(t),t)+2c(t).
\end{align}

Since $\cV$ is strongly convex following Assumption \ref{assumption_approximation}, we have
\[\cV(x,t)\geq \cV(\hat x(t),t)+\frac{\gamma}{2}\|\hat x(t)-x\|^2\]

\[\cV_\Phi(x,t)\geq \cV_\Phi(\hat x(t),t)+\frac{\gamma_\Phi}{2}\|\hat x_\Phi(t)-x\|^2\]
Hence,

\[\|\hat x_\Phi(t)-\hat x(t)\|\leq \|\hat x_\Phi(t)-x\|+\|\hat x(t)- x\| \]
\[\leq\sqrt{ \frac{2}{\gamma_\Phi}}|\cV_\Phi(x,t)-\cV_\Phi(\hat x_\Phi(t),t)|^{\frac{1}{2}}+\sqrt{\frac{2}{\gamma}}|\cV(x,t)-\cV(\hat x(t),t)|^{\frac{1}{2}}\]
\[=2\sqrt{\frac{c(t)}{\gamma_{min}}}\]

where $\gamma_{min}=\min\{\gamma,\gamma_\Phi\}$. 
\qed
\end{proof}

\begin{remark}
The curvature constants, denoted by $\gamma_1$ and $\gamma_\Phi$, can be directly computed as the minimum eigenvalues of the Hessian matrices associated with the value functions $\cV$
 and $\cV_\Phi$ respectively. Furthermore, the constant $\delta_h$ used in Assumption \ref{assumption_approximation} will decrease as we use higher order basis function in the approximation of the output mapping $h(x)$ as shown in our next Theorem \ref{thm:lifted_bounds}.  
 \end{remark}

 \subsection{KMD and Lifting of Output}\label{section_KMD}
The principal eigenfunctions of the Koopman operator satisfy the algebra property stated in Property~\ref{property1}. This algebraic structure can be exploited to construct improved projections of the output map $
h(x)$ in cases where 
$h(x)$ does not lie in the span of the principal eigenfunctions alone i.e., when Assumption~\ref{assume_span} is not satisfied. To this end, we define multi-index $\alpha = (\alpha_1,\ldots,\alpha_n)\in\mathbb{N}^n$, define

\[
\phi^\alpha(x) := \prod_{i=1}^n \phi_i(x)^{\,\alpha_i}.
\]

Assume that each lifted eigenfunction $\phi^\alpha(x)$ belongs to $\mathcal{C}^1(\mathbb{R}^n)$.  
Let $\mathcal{A}\subset \mathbb{N}^n$ denote the collection of multi-indices used in the approximation (e.g., all $\alpha$ with $|\alpha|\le k$).  
We define the lifted principal--eigenfunction vector
\[
\bar{\Phi}(x) := \big(\phi^\alpha(x)\big)_{\alpha\in\mathcal{A}}
\in \mathbb{R}^N,
\]
where $N = |\mathcal{A}|$ is the cardinality of $\mathcal{A}$.  
The over-bar notation indicates an augmented set of Koopman eigenfunctions.  
The Koopman eigenvalue associated with $\phi^\alpha(x)$ is 
\[
\sum_{i=1}^n \alpha_i \lambda_i,
\]
where $\lambda_i$ is the eigenvalue corresponding to the principal eigenfunction $\phi_i(x)$.  
Both $\bar{\Phi}(x)$ and $\bar{\Lambda}$ admit the block structure
\begin{align}
\bar{\Phi}(x)=
\begin{pmatrix}
\Phi(x)\\[3pt]
q_1(\Phi(x))
\end{pmatrix},
\qquad
\bar{\Lambda}=
\begin{pmatrix}
\Lambda & 0\\[3pt]
0 & q_2(\Lambda)
\end{pmatrix},
\label{bardefinition}
\end{align}
where $q_1:\mathbb{R}^n\to\mathbb{R}^{N-n}$ and  
$q_2:\mathbb{R}^n\to\mathbb{R}^{(N-n)\times(N-n)}$,  
assuming that $\Phi$ and $\Lambda$ are in real canonical form.

Instead of expanding the output observable $h(x)$ solely in terms of the principal eigenfunctions $\Phi(x)$, we may perform a Koopman Mode Decomposition (Definition~\ref{definition_kmd}) of $h(x)$ using the lifted eigenfunctions $\bar{\Phi}(x)$.  
Repeating the least-squares optimization steps in \eqref{optimization_prob}--\eqref{Gb_definition} with $\Phi(x)$ replaced by $\bar{\Phi}(x)$ yields coefficients $\bar{c}_i\in\mathbb{R}^N$.  
Since $\bar{\Phi}(x)$ contains $\Phi(x)$ as a subset (Eq.\ \eqref{bardefinition}), the projection error satisfies
\[
\left\langle h-\bar{c}_i^\top \bar{\Phi},\,h-\bar{c}_i^\top \bar{\Phi}\right\rangle
\;\le\;
\left\langle h-c_i^\top \Phi,\,h-c_i^\top \Phi\right\rangle.
\]

The main results of this section show that the optimal value function $\mathcal{V}(x,t)$, defined as the solution of the Hamilton--Jacobi equation \eqref{HJequationPhi}, admits computable upper and lower bounds in terms of quadratically parameterized value functions expressed in the lifted eigenfunction coordinates. We make following assumption.
\begin{assumption}[KMD of the output]\label{assume_spanhigher} We assume that the output mapping $h(x)$ lies in the span of the lifted principal eigenfunctions i.e., there exists a matrix $\bar C\in \mR^{N\times p}$ such that 
\begin{align}
h(x)=\bar C\bar \Phi
\end{align}
\end{assumption}
\begin{remark} If this assumption is not satisfied then the results along the lines of Theorem \ref{theorem2} can be proved for this case of lifted eigenfunctions. 
\end{remark}
To prove the next main results of this paper, we derive the following bounds on the following matrix function of $x$ of interest
\begin{align}
\bar R(x)
:=\frac{\partial \bar{\Phi}}{\partial x}R_1(x)^{-1}\frac{\partial \bar{\Phi}}{\partial x}^\top =
\frac{\partial \bar{\Phi}}{\partial x}
\left(
\frac{\partial \Phi}{\partial x}^\top
R^{-1}
\frac{\partial \Phi}{\partial x}
\right)^{-1}
\frac{\partial \bar{\Phi}}{\partial x}^\top\label{definitionbarR}
\end{align}
Using the structure of lifted eigenfunction $\bar \Phi(x)$ in (\ref{bardefinition}), we have
\[
\frac{\partial \bar\Phi}{\partial x}(x)
=
\begin{pmatrix}
I_n \\[3pt] J_q(\Phi(x))
\end{pmatrix}
J_\Phi(x),
\]
where $J_\Phi(x)=\frac{\partial \Phi}{\partial x}$ and  $J_q(z)=\tfrac{\partial q}{\partial z}(z)$.
Define
\[
B(x)
:= 
J_{\bar\Phi}(x)\,J_\Phi(x)^{-1}
=
\begin{pmatrix}
I_n \\[3pt]
J_q(\Phi(x))
\end{pmatrix}
\in\mathbb{R}^{N\times n}.
\]
For a symmetric positive definite matrix $R\succ 0$, the matrix of interest can be written as
\[
\bar R(x)
=
\frac{\partial \bar\Phi}{\partial x}
\Big(\frac{\partial \Phi}{\partial x}\Big)^{-1}
R
\Big(\frac{\partial \Phi}{\partial x}\Big)^{-T}
\Big(\frac{\partial \bar\Phi}{\partial x}\Big)^{\!\top}
=
B(x)\,R\,B(x)^{\!\top}.
\]

Since for any $w\in\mathbb{R}^n$,
\[
\|B(x)w\|^2
= 
\|w\|^2 + \|J_q(\Phi(x))\,w\|^2
\ge \|w\|^2,
\]
the smallest singular value of $B(x)$ satisfies
\[
\sigma_{\min}(B(x)) \ge 1.
\]
Similarly,
\[
\|B(x)w\|^2 
\le \big(1 + \|J_q(\Phi(x))\|^2\big)\,\|w\|^2,
\]
which implies
\[
\sigma_{\max}(B(x))^2 
\le 1 + \|J_q(\Phi(x))\|^2.
\]

Let $0<\lambda_{\min}(R)\le \lambda_{\max}(R)$ denote the extremal eigenvalues of $R$.
Using the standard bounds for $B R B^\top$, we obtain the following matrix inequality
\[
\lambda_{\min}(R)\,\sigma_{\min}(B(x))^2\,I_N
\;\preceq\;
\bar R(x)
\;\preceq\;
\lambda_{\max}(R)\,\sigma_{\max}(B(x))^2\,I_N.
\]
Substituting the bounds on the singular values of $B(x)$ yields the estimate
\[
\lambda_{\min}(R)\,I_N
\;\preceq\;
\bar R(x)
\;\preceq\;
\lambda_{\max}(R)\,\big(1 + \|J_q(\Phi(x))\|^2\big)\,I_N.
\]

% Let $\|J_q(\Phi(x))\|\le L$ on a set $\Omega\subseteq \mR^n$, then we get  the following uniform bounds
% \[
% \lambda_{\min}(R)\,I_N
% \;\preceq\;
% \bar R(x)
% \;\preceq\;
% \lambda_{\max}(R)\,(1 + L^2)\,I_N.
% \]
The bounds from the above discussion can be summarized in the following assumption. 
\begin{assumption}\label{assumption_boundsbarR} We assume that the 
satisfies the uniform bounds
\begin{equation}
\label{eq:R_bounds_bar}
\bar R_{-}
\;\leq\;
\bar R(x)
\;\leq\;
\bar R_{+},
\end{equation}
\begin{equation}
\label{eq:Rpm_def}
\bar R_{+} = \lambda_{\max}(R)\,(1+L^2)\,I_{N},
\qquad
\bar R_{-} = \lambda_{\min}(R)\,I_{N},
\end{equation}
for some constant $L \ge 0$ such that
\begin{equation}
\label{eq:q1_grad_bound}
\left\|
\frac{\partial q_1(\Phi(x))}{\partial x}
\right\|
\leq L
\quad\text{for all } x\in\mathbb{R}^n.
\end{equation}
\end{assumption}
\begin{theorem}
\label{thm:lifted_bounds}
Let $\mathcal{V}(x,t)$ denote the optimal value function associated with the Hamilton--Jacobi equation (\ref{HJequationPhi}). Let $\bar \Phi(x)$ be the lifted eigenfunctions  as constructed in~\eqref{bardefinition}. We 
assume that:
\begin{enumerate}[]
\item Each lifted eigenfunction $\phi^\alpha(x)$ belongs to $\mathcal{C}^1(\mathbb{R}^n)$.
\item The matrix $\bar R(x)$ from Eq. (\ref{definitionbarR}) satisfies Assumption \ref{assumption_boundsbarR}
% \[
% \bar R(x)
% :=
% \frac{\partial \bar{\Phi}}{\partial x}
% \left(
% \frac{\partial \Phi}{\partial x}^\top
% R^{-1}
% \frac{\partial \Phi}{\partial x}
% \right)^{-1}
% \frac{\partial \bar{\Phi}}{\partial x}^\top
% \]

\end{enumerate}
% where $q_1$ is defined in (\ref{bardefinition}). 
Define the lifted, quadratically parameterized value functions
\begin{equation}
\label{eq:Vpm_bar_def}
\mathcal{V}_{\bar{\Phi}}^{\pm}(x,t)
=
\frac{1}{2}\bar{\Phi}(x)^\top P^{\pm}(t)\,\bar{\Phi}(x)
+ s^{\pm}(t)^\top \bar{\Phi}(x)
+ r^{\pm}(t),
\end{equation}
where the coefficient triples $(P^{\pm}(t),s^{\pm}(t),r^{\pm}(t))$ solve the Riccati-type systems
\begin{align}
\dot{\bar P}^{\pm}(t)
&+
\bar{\Lambda}^\top \bar P^{\pm}(t)
+
\bar P^{\pm}(t)\,\bar{\Lambda}
+
\bar P^{\pm}(t)\,\bar R_{\pm}\,\bar P^{\pm}(t)
\nonumber\\&-
\bar C^\top Q^{-1}\bar C
=0,
\label{eq:Ppm_ode}
\\
\dot{\bar s}^{\pm}(t)
&+
\bar{\Lambda}^\top \bar s^{\pm}(t)
+
\bar P^{\pm}(t)\,\bar R_{\pm}\,\bar s^{\pm}(t)
+
\bar C^\top Q^{-1} y(t)
=0, \label{eq:spm_ode} \\
\dot{\bar r}^{\pm}(t)
&+
\frac{1}{2}(\bar s^{\pm}(t))^\top\,\bar R_{\pm}\,\bar s^{\pm}(t)
-
\frac{1}{2}y^\top(t) Q^{-1} y(t)
=0.
\label{eq:rpm_ode}
\end{align}

Then, for initial conditions $\cV^+_{\bar \Phi}(x,0)=\cV^-_{\bar \Phi}(x,0)=\cV(x,0)$, the functions $\mathcal{V}_{\bar{\Phi}}^{-}(x,t)$ and $\mathcal{V}_{\bar{\Phi}}^{+}(x,t)$ satisfy the sandwich bounds
\begin{equation}
\label{eq:V_sandwich}
\mathcal{V}_{\bar{\Phi}}^{+}(x,t)
\;\leq\;
\mathcal{V}(x,t)
\;\leq\;
\mathcal{V}_{\bar{\Phi}}^{-}(x,t),
\qquad \forall\,x\in\mathbb{R}^n,\; t\in[0,T],
\end{equation}
where the inequalities are understood pointwise.
\end{theorem}
\begin{proof}
The proof proceeds by constructing sub and supersolutions of the Hamilton--Jacobi equation~\eqref{HJequationPhi} in the lifted eigenfunction coordinates and then invoking a comparison principle. We will only prove the inequality $\cV\geq \cV_{\bar \Phi}^+$, the inequality in the other direction follows along similar lines. Substituting the expression of $\cV_{\bar \Phi}^+$ in the HJ equation (\ref{HJequationPhi}) and using Assumption \ref{assume_spanhigher} along with bounds in (\ref{eq:R_bounds_bar}) we obtain
\[\frac{\partial \cV^+_{\bar \Phi}}{\partial t}+H(x,\nabla_x \cV_{\bar \Phi}^+,t)=\frac{1}{2}\bar \Phi^\top \dot P^+\bar\Phi+\dot \bar s^+(t)^\top \bar \Phi+\dot r^+\]
\[\bar \Phi^\top \bar P^+\bar \Lambda \bar \Phi+\bar \Phi^\top\bar \Lambda^\top \bar P^+\bar \Phi+\bar s^+(t)^\top \bar \Lambda \bar \Phi\]\[ +\frac{1}{2}\left(\bar \Phi^\top \bar P^++\bar s^\top\right)\frac{\partial \bar \Phi}{\partial x}R_1(x)^{-1}\frac{\partial \bar \Phi}{\partial x}^\top \left(\bar \Phi^\top \bar P^++\bar s^\top\right)^\top \]
\[-\frac{1}{2}(y-\bar C\bar \Phi)^\top Q^{-1}(y-\bar C\bar \Phi)\]
\[\leq  \frac{1}{2}\bar \Phi^\top \dot P^+\bar\Phi+\dot \bar s^+(t)^\top \bar \Phi+\dot r^++\bar \Phi^\top \bar P^+\bar \Lambda \bar \Phi+\bar \Phi^\top\bar \Lambda^\top \bar P^+\bar \Phi\]
\[+\bar s^+(t)^\top \bar \Lambda \bar \Phi +\frac{1}{2}\left(\bar \Phi^\top \bar P^++\bar s^\top\right)\bar R_+ \left(\bar \Phi^\top \bar P^++\bar s^\top\right)^\top \]
\[-\frac{1}{2}(y-\bar C\bar \Phi)^\top Q^{-1}(y-\bar C\bar \Phi)=0\]
The above equality eqaul to zero follows because $(\bar P^+,\bar s^+, \bar r^+)$ satisfy the Riccatti equation (\ref{eq:rpm_ode}). Hence we get,
\[\frac{\partial \cV^+_{\bar \Phi}}{\partial t}+H(x,\nabla_x \cV_{\bar \Phi}^+,t)\leq 0 \]
The results, $\cV_{\bar \Phi}^+(x,t)\leq \cV(x,t)$, then follows using comparison principle as $\cV^+_{\bar \Phi}(x,t)$ is the viscosity sub-solution and $\cV(x,t)$ is the assumed viscosity solution (Assumption \ref{assumption_HJviscositysolution}) and  $\cV_{\bar \Phi}^+(x,0)= \cV(x,0)$\cite{CrandallIshiiLions1992}. The proof for $\cV_{\bar \Phi}^-(x,t)\geq \cV(x,t)$, will follow along similar lines. 
\end{proof}

\subsection{ Relation to the posterior density, Zakai/Mortensen viewpoints, and the Kalman--Bucy filter}

The optimal value function $\mathcal V_\Phi(x,t)$ obtained from the
Hamilton--Jacobi (HJ) formulation admits an immediate probabilistic
interpretation. In particular, it induces an \emph{unnormalized} posterior
density of the form
\[
\bar p(x,t \mid y_{[0,t]}) \;\propto\; \exp\!\big(-\mathcal V_\Phi(x,t)\big).
\]
Consequently, the KBK estimate returned by the proposed filter is a
maximum-a-posteriori (MAP) estimate,
\[
\hat x(t)=\arg\min_x \mathcal V_\Phi(x,t)
       =\arg\max_x \bar p(x,t\mid y_{[0,t]}),
\]
which coincides with the maximum-likelihood/modal trajectory estimate
associated with the Mortensen minimum-energy formulation. Normalizing
$\bar p$ yields the (approximate) posterior density
\[
p(x,t)=\frac{\bar p(x,t)}{\int_{\mathbb R^n}\bar p(x,t)\,dx}.
\]

This perspective also clarifies the relationship between classical density-based
filters and the proposed spectral approach. The Zakai/Kushner--Stratonovich
filters propagate the (unnormalized/normalized) posterior density directly
through a (stochastic) Fokker--Planck evolution, whereas the Mortensen/HJ
formulation propagates the negative log-density (value function) whose minimizer
yields the MAP estimate. The proposed KBK filter leverages Koopman spectral
coordinates specifically, principal Koopman eigenfunctions to parameterize
this value function in a tractable quadratic form, thereby producing a
finite-dimensional, Riccati-type evolution for nonlinear estimation.

Finally, the proposed KBK filter reduces to the classical Kalman--Bucy filter
for linear time-invariant (LTI) systems. For $\dot x=Ax+w$ and $y=Cx+v$, the
principal Koopman eigenfunctions are available analytically as
$\Phi(x)=W^\top x$, where $W^\top A=\Lambda W^\top$. In this case, the KBK
filter equations in (\ref{estimator}) coincide with the Kalman--Bucy equations
expressed in the transformed coordinates $\Phi$.

\subsection{Bias and Cramér-Rao considerations.}
The proposed Kalman--Bucy--Koopman (KBK) filter is formulated through the
Mortensen minimum–energy principle and therefore computes a
maximum–a–posteriori (MAP) state estimate by minimizing an approximate
Hamilton--Jacobi value function
$V(x,t) \approx -\log p(x\mid \mathcal Y_t)$.
Consequently, the estimator targets the \emph{mode} of the posterior
distribution rather than its conditional mean.
Unlike minimum–mean–square–error (MMSE) estimators, MAP estimators are
generally biased whenever the posterior density is non–Gaussian or
asymmetric, which is typical for nonlinear dynamical systems.
Accordingly, unbiasedness and classical Cramér-Rao lower bound (CRLB)
optimality are not expected in general, since the CRLB applies primarily to
unbiased estimators of the conditional mean.

Instead, the KBK filter should be interpreted as a deterministic
value-function approximation to the Mortensen optimal control problem,
whose accuracy is governed by the quality of the Koopman spectral
approximation.  In regimes where the posterior is approximately Gaussian
(e.g., small noise or near-linear dynamics), the posterior mean and mode
coincide, and the KBK estimate empirically approaches MMSE performance,
often yielding error covariances close to the posterior Cramér-Rao bound.
% More generally, uncertainty can be characterized through the local curvature
% of the value function via a Laplace approximation,
% $P(t) \approx [\nabla^2 V(\hat x(t),t)]^{-1}$, which provides a principled
% measure of estimator confidence.  

\section{Computation of the Koopman Eigenfunctions}

For the propagation of maximum-likelihood estimates, it is essential to obtain the Hessian information of the value function, which in turn depends on the principal eigenfunctions of the Koopman operator. In \cite{vaidya2025koopman}, we have shown that these principal eigenfunctions satisfy a linear partial differential equation (PDE) and proposed a Galerkin projection–based method for their finite-dimensional approximation. Algorithmically, the computation of the principal Koopman spectrum can be formulated as a least-squares optimization problem. However, similar to the finite-dimensional approximation of the Koopman operator itself, a key challenge in this procedure lies in the selection of an appropriate set of basis functions. The accuracy and efficiency of the resulting approximation are highly sensitive to this choice, making the design of data-efficient and adaptive basis selection strategies an important open problem.

In what follows, we present two computational approaches for approximating the principal eigenfunctions, both grounded in a method-of-characteristics–based solution of the linear PDE satisfied by the eigenfunctions. The first approach, presented in Subsection~\ref{section_pathintegral}, derives an explicit path-integral representation for the eigenfunctions. This representation reveals that the principal eigenfunctions can be computed directly from system trajectories, without requiring discretization of the associated PDE over the state space. Consequently, the computation reduces to trajectory-based integration along characteristic curves, making the approach inherently compatible with data-driven implementation and scalable to higher-dimensional systems. The second approach, described in Subsection~\ref{section_basisPI}, introduces a dynamics-informed basis construction motivated by the method of characteristics, enabling efficient eigenfunction approximation within a finite-dimensional function space while retaining alignment with the underlying flow geometry.

\subsection{Path-Integral Approach for Eigenfunction Computation}\label{section_pathintegral}
% The following assumption is introduced on the system dynamics to facilitate the computation of the Koopman principal spectrum.
% \begin{assumption}\label{assumption_hyperbolic} We assume that the system dynamics $\dot x=f(x)$ has a hyperbolic equilibrium point which is assumed to be at the origin.
% \end{assumption}

Following Assumption \ref{assume_peig}, the nonlinear system admits a decomposition into linear and nonlinear parts as follows.
\begin{align}
\dot x=f(x)=Ax+F_n(x)\label{system}
\end{align}

The principal eigenfunction of the Koopman operator admits the decomposition of the form 
\[(\phi_{\lambda_1},\ldots,\phi_{\lambda_n})^\top=\Phi(x)=W^\top x+H(x)\]
where $W^\top x$ is the linear and $H(x)=(h_{\lambda_1}(x),\ldots, h_{\lambda_n}(x))\in \mR^n$ is the purely nonlinear part and hence satisfies $\frac{\partial H}{\partial x}(0)=0$. The individual eigenfunction admits a decomposition into linear and nonlinear parts as 
\[\phi_\lambda(x)=w_\lambda^\top x+h_\lambda(x)\]
Substituting the above expression of the eigenfunction in $\frac{\partial \phi_\lambda}{\partial x}f(x)=\lambda \phi_\lambda(x)$, we obtain after comparing the linear and nonlinear parts that $w_\lambda$ and $h_\lambda$ should satisfy
% $W^\top$ satisfies
% \[W^\top A=\Lambda W^\top\]
% where $\Lambda$ is the matrix of eigenvalues of $A$. Let, $h_\lambda(x)$, be the nonlinear part of the eigenfunction  associated with the eigenvalue $\lambda$  assumed to be real. We will comment on the case of the complex eigenvalue later. 
% The nonlinear part $h_\lambda(x)$ satisfies following linear PDE.
\begin{align}
w_\lambda^\top A=\lambda w_\lambda^\top,\;\;\;\frac{\partial h_\lambda}{\partial x}f(x)-\lambda h_\lambda(x)+w_\lambda^\top F_n(x)=0\label{linearPDE}
\end{align}
So $w_\lambda^\top$ is the left eigenvector of $A$ corresponding to the eigenvalue $\lambda$. The solution to this linear PDE is given by the method of characteristics and is of the form
\begin{align}
h_\lambda(x)=e^{-\lambda t}h_\lambda(s_t(x))+\int_0^t e^{-\lambda \tau}w_\lambda^\top F_n(s_\tau(x))d\tau\label{methodofcharacteristicsformula}
\end{align}
% We observe that the nonlinear component of the principal eigenfunction satisfies a linear partial differential equation (PDE). Several methodologies can be employed to approximate this eigenfunction. In \cite{vaidya2025koopman}, a Galerkin projection framework was developed to approximate the solution of the linear PDE within a finite-dimensional subspace. In \cite{deka2023path}, an alternative path-integral formulation was introduced, derived from the method of characteristics, to approximate the principal eigenfunction. Under the assumption that the equilibrium point at the origin of system~(\ref{system}) is asymptotically stable, an explicit expression for the principal eigenfunction can be obtained, as established in Theorem~\ref{theorem_pi}, which generalizes the results reported in \cite{deka2023path}.

The path-integral formulation is particularly attractive, as it eliminates the need for basis-function expansions. Furthermore, in the propagation of maximum-likelihood estimates (refer to the dynamics of the KBK filter), the local behavior of the principal eigenfunction along the estimated trajectory~$\hat{x}(t)$ is of primary interest. The path-integral representation is well suited for this purpose, enabling local computation of the eigenfunction without significant computational overhead. Although the expression for the nonlinear component of the eigenfunction in Eq.~(\ref{methodofcharacteristicsformula}) remains implicit—requiring knowledge of the boundary condition—the following theorem provides an explicit formula for the computation of the principal eigenfunction under the stated stability assumption. The following theorem generalizes the results from \cite{deka2023path} for the computation of principal eigenfunctions.

\begin{theorem}\label{theorem_pi} For the dynamical system (\ref{system}) satisfying Assumption \ref{assume_peig}, let the equilibrium point at the origin be asymptotically stable with domain of attraction  $\cD$. Let the nonlinear terms $F_n(x)$ contains nonlinearity of order $\ell$ i.e., 
\begin{align}F_n(x)\approx O(\|x\|^\ell).\label{nonlineartermorder}
\end{align}
where $\ell\geq 2$. If the eigenvalue, $\lambda$, at the linearization satisfy 
\begin{align}
-{\rm Re}(\lambda)+\ell{\rm Re}(\lambda_{max})<0\label{eigenspread}
\end{align}
where $\lambda_{max}$ is the eigenvalue of $A$ with maximum real part. Then the principal eigenfunction corresponding to the eigenvalue $\lambda$ is given by
\begin{align}
\phi_\lambda(x)=w_\lambda^\top x+\int_0^\infty e^{-\lambda t}w_\lambda^\top F_n(s_t(x))dt\label{formulaeigenfunction}
\end{align}
for all $x\in \cD$. 
\end{theorem}
\begin{proof}
We divide the proof in two parts. We start with the  following condition  
\begin{align}
\lim_{t\to \infty} e^{-\lambda t}F_n(s_t(x))=0\label{assumeconvergence}
\end{align}
for all $x\in \cD$.
Using this condition (\ref{assumeconvergence}), we first show that $h_\lambda(x)$ of the form 
\begin{align}h_\lambda(x)=\int_0^\infty e^{-\lambda_k t}w_\lambda^\top F_n(s_t(x))dt\label{assumesolution}
\end{align}
will satisfy the linear PDE (\ref{linearPDE}). In the second part of the proof, we show that the condition (\ref{assumeconvergence}) is satisfied under the assumption stated in the theorem statement i.e., (\ref{eigenspread}). We begin with the first part of the proof. Using the assumed form of the solution for $h_\lambda(x)$ from (\ref{assumesolution}) we obtain
\begin{align}
\frac{\partial h_\lambda}{\partial x}f(x)=\int_0^\infty e^{-\lambda t} \frac{\partial w_\lambda^\top F_n(s_t(x))}{\partial x}f(x)dt \nonumber\\
=\int_0^\infty e^{-\lambda t} \frac{d}{dt} (w_\lambda^\top F_n(s_t(x)))dt \nonumber\\
\left.e^{-\lambda t}w_\lambda^\top F_n(s_t(x))\right|_{t=0}^\infty+\lambda\int_0^\infty e^{-\lambda t}w^\top F_n(s_t(x))dt
\end{align}
Using the assumption (\ref{assumeconvergence}), we obtain
\[\frac{\partial h_\lambda}{\partial x}f(x)=-w_\lambda^\top F_n(x)+\lambda h_\lambda(x)\]
which is precisely the linear PDE (\ref{linearPDE}). 

For the second part of the proof, we will show that if the eigenvalues satisfy the spread condition (\ref{eigenspread}), then (\ref{assumeconvergence}) will hold true. Since the equilibrium point is assumed to be asymptotically stable we know that for given $\delta>0$ there exists a time $T(\delta)$ such that $\|s_t(x)\|\leq \delta$ for $t\geq T$ and $\lim_{t\to \infty}s_t(x)=0$. Following (\ref{nonlineartermorder}), we know that there exists a constant $K>0$ and $\delta>0$ such that
\[\|F_n(x)\|\leq K\|x\|^\ell,\;\;\;\|x\|\leq \delta\]
Hence,
\[\| F_n(s_t(x))\|\leq K \|s_t(x)\|^\ell\]
for $t\geq T$. 
Now for $\|x\|\leq \delta$, there exists, by Hartman Grobman theorem, a near identity change of coordinates with inverse in the small neighborhood around the origin, say of size $\|x\|\leq \delta$, of the form
\begin{equation}z=x+d(x)=D(x)\iff x=D^{-1}(z)=z+\bar d(z),\label{HG}
\end{equation}
with $d(x)$ and $\bar d(z)$ purely nonlinear
such that the nonlinear system is transformed \textcolor{black}{into} linear system i.e.,
$\dot x=Ax+F_n(x)\implies \dot z=Az$ 
and hence 
\[s_t(x)=D^{-1}(e^{A t}D(x))\implies
s_t(x)=D^{-1}(e^{A t}(x+d(x)))\]
% \[=e^{\bA t}(\bx+\bd(\bx))+\bar \bd(e^{\bA t}(\bx+\bd(\bx)))\]
% \[=e^{\bA t}\bx+e^{\bA t}\bd(\bx)+\bar \bd(e^{\bA t}(\bx+\bd(\bx)))\]
\[=e^{A t}x+e^{A t}d(x)+\bar d(e^{A t}x+e^{A t}d(x)).\]
In the above, we have used (\ref{HG}) for $D^{-1}$. Since $\bar d(z)$ is purely nonlinear, for $\|x\|\leq \delta$, we can get using mean value theorem 
\[\|\bar d(z)\|\leq c_{\bar d} \|z\|^2,\;\;\;\;\|d(x)\|\leq c_d \|x\|^2.\]
for some constants $c_d$ and $c_{\bar d}$.  Using the above inequality, Cauchy Schwartz inequality,  and the fact that $\|x\|\leq \delta$, we obtain 
% \[\|\bs_t(\bx)\|\leq \|e^{\bA t}\bx\|+c_{\bd}\|e^{\bA t}\|\|\bx\|^2+c_{\bar \bd}\|e^{2\bA t}\|\|(\bx+\bd(\bx))\|^2\]
% \[\leq \|e^{\bA t}\|\|\bx\|+c_\bd \|e^{\bA t}\|\|\bx\|^2+c_{\bar \bd}\|e^{2\bA t}\|\|\bz\|^2\]
\[\|s_t(x)\|\leq c_1 e^{{\rm Re}(\lambda_{max} t)}\implies \|s_t(x)\|^\ell\leq c_1^\ell e^{{\rm Re}(\ell\lambda_{max} t)} \]
for some constant $c_1$ that depends on $\epsilon, c_d$, and $\bar c_d$.

\end{proof}

\begin{remark}Note that by time-reversing the vector field, the above theorem can also be applied for the case when the equilibrium point is anti-stable i.e., matrix $-A$ (Eq. \ref{system}) is Hurwitz. 
\end{remark}

For system with saddle-type equilibrium point, the path-integral formula can be used for the computation of the eigenfunction under the following assumption.

\begin{assumption}\label{assumption_saddle} We assume that the nonlinear parts of the eigenfunctions are bounded i.e.,
\begin{align}
|h_\lambda(x)|\leq M\label{eigenfunction_bounded}
\end{align}
for some constant $M$ and for all $x\in \mR^n$. 
\end{assumption}
\begin{theorem} Under the Assumption \ref{assumption_saddle}, the Koopman principal eigenfunction corresponding
to eigenvalue $\lambda$
\begin{itemize}
\item with $ {\rm Re}(\lambda)>0$ is given by
\begin{align}
\phi_\lambda(x)=w_\lambda^\top x+\int_0^\infty e^{-\lambda t} w_\lambda^\top F_n(s_t(x))dt \label{saddle_eigenfunctionformula}
\end{align}
\item with $ {\rm Re}(\lambda)<0$ is given by
\begin{align}
\phi_\lambda(x)=w_\lambda^\top x-\int_0^\infty e^{\lambda \theta} w_\lambda^\top F_n(s_{-\theta}(x))d\theta \label{saddle_eigenfunctionformula2}
\end{align}
\end{itemize}
\end{theorem}
\begin{proof} The method of characteristics based solution to the nonlinear part of the principal eigenfunction satisfying thelinear PDE is given by (\ref{methodofcharacteristicsformula}). Now consider any eigenvalue with ${\rm Re}(\lambda>0$. We have using (\ref{eigenfunction_bounded})
\[\lim_{t\to \infty}|e^{-\lambda t}h_\lambda(s_t(x))|\leq e^{-\lambda t}M=0\]
Hence, using (\ref{methodofcharacteristicsformula}) in the limit as $t\to \infty$, we obtain the desired expression (\ref{saddle_eigenfunctionformula}) for the eigenfunction. 
For ${\rm Re}(\lambda)<0$, we can reverse the time. In particular, we have  

\[\lim_{t\to -\infty}|e^{-\lambda t}h_\lambda(s_t(x))|\leq e^{-\lambda t}M=0\]
and hence (\ref{methodofcharacteristicsformula}) in the limit as $t\to -\infty$ reduces to 
\begin{align}
h_\lambda(x)=\int_0^{-\infty} e^{-\lambda \tau}w_\lambda^\top F_n(s_\tau(x))d\tau
\end{align}
We obtain the desired formula in (\ref{saddle_eigenfunctionformula2}) after performing the change of variable from $t=-\theta$. 
% now $\tau=-\theta, d\tau=-d\theta$ and $\tau=-\infty, \theta=\infty$
% \begin{align}
% h_\lambda(x)=-\int_0^\infty e^{\lambda \theta}w_\lambda^\top F_n(s_{-\theta}(x))d\theta
% \end{align}

\end{proof}

\begin{remark} The formulas (\ref{formulaeigenfunction}) and (\ref{saddle_eigenfunctionformula}) generalize to the case of complex eigenvalues where one can treat the real and imaginary parts of the complex eigenfunctions separately. 
\end{remark}
The infinite horizon path-integral formula discussed in the previous section could be restrictive. These restriction arise in the form of conditions (\ref{eigenspread}) and (\ref{eigenfunction_bounded}). In \cite{deka2024pathcdc}, the finite-time extension of the path integral formula is discussed. However, in the following section, we use the finite-horizon path-integral formula and use it to construct a dynamically informed basis functions which are used towards the approximation of the principal eigenfunctions. 

\subsection{Characteristics-Inspired Basis Functions for Koopman Eigenfunction Approximation}\label{section_basisPI}

A central challenge in approximating Koopman eigenfunctions for nonlinear systems is the construction of basis functions that faithfully capture the geometry of the underlying dynamics. Classical choices such as polynomial, trigonometric, or radial basis functions are agnostic to the system vector field $f(x)$, and consequently require a large number of elements to resolve even moderately nonlinear behavior. Approximation quality often degrades rapidly away from the origin or outside the training region.
To overcome these limitations, we introduce a \emph{dynamics-informed} basis construction inspired by the \emph{method of characteristics} for the Koopman eigenfunction PDE
\begin{equation}
\frac {\partial \phi_\lambda(x)}{\partial x} f(x) = \lambda \phi_\lambda(x),
\label{eq:koopman-eig-pde}
\end{equation}
whose characteristic curves coincide with trajectories of the nonlinear flow $s_t(x)$. This relation motivates the design of \emph{characteristics-based basis functions} that incorporate the local behavior of the nonlinear flow directly into the eigenfunction approximation.

\subsubsection*{Path-Integral Characteristics Basis}

Given a candidate Koopman eigenvalue $\lambda$ and a left eigenvector $w$ of the linearization matrix $A$, we define a family of nonlinear basis functions
\begin{equation}
\psi_k^\lambda(x) 
= \int_0^{\Delta_k} e^{-\lambda t} \, w^\top F_n(s_t(x)) \, dt,
\qquad k = 1,\dots,K,\label{stepsintegration}
\end{equation}
where $\Delta_k$ are short, increasing integration horizons that encode multi-scale nonlinear effects.
Each function $\psi_k^\lambda(x)$ is therefore a localized probe of the nonlinear dynamics, capturing how the nonlinear component of the vector field distorts trajectories over different time scales. Unlike classical static dictionaries, these features explicitly incorporate the system’s trajectory geometry, yielding a basis that is:

\begin{itemize}
    \item \textit{Dynamics-adaptive:} the basis aligns with the characteristic curves of the eigenfunction PDE \eqref{eq:koopman-eig-pde};
    \item \textit{State-localized:} since $\psi_k^\lambda(x)$ is computed from short trajectories initialized at $x$, no global trajectory integration is required;
    \item \textit{Multi-scale:} the distinct horizons $\Delta_k$ allow the basis to capture both local and moderately global nonlinear behavior;
    \item \textit{Linear-term free:} because $F_n$ excludes the linear component, the nonlinear basis functions satisfy $\psi_k^\lambda(x) = \mathcal{O}(\|x\|^2)$, guaranteeing compatibility with the decomposition $\phi_\lambda(x) = w^\top x + h_\lambda(x)$.
\end{itemize}

\subsubsection*{Galerkin Approximation Using Characteristics Basis}

Using the characteristics-based dictionary $\{\psi_k^\lambda\}$, the nonlinear part of the Koopman eigenfunction is approximated by
\begin{equation}
h_\lambda(x) \approx \sum_{k=1}^K a_k \psi_k^\lambda(x),
\end{equation}
and substituting $h_\lambda(x)$ into \eqref{linearPDE} yields a finite-dimensional least-squares problem for the coefficients $a_k$. We omit the details on the finite-dimensional Galerkin based approximation of eigenfunctions leading to the least-squares problem as it is a standard approach discussed in various literature including \cite{vaidya2025koopman} in the specific context of the Koopman principal eigenfunctions. This produces a Galerkin-type approximation scheme in which both test and trial functions are derived from the flow geometry of the system.
The proposed construction offers several key benefits:
\begin{itemize}
    \item \textit{Flow-aligned representation:} the basis adapts naturally to the nonlinear trajectory structure of the system.
    \item \textit{Compactness:} because each $\psi_k(x)$ carries dynamical information, only a small number of basis functions are required for accurate approximation.
    \item \textit{Scalability:} no combinatorial blow-up, unlike polynomial lifting.
    \item \textit{Online compatibility:} the local nature of the basis allows computation along an estimated trajectory, enabling real-time implementation. 
    \item \textit{Theoretical grounding:} the construction follows directly from the characteristic representation of eigenfunction solutions, providing a principled alternative to heuristic dictionary learning.
\end{itemize}

\section{Simulation Results}

In this section, we present simulation results for the spectral Koopman-based design of nonlinear estimator. All the simulation codes are developed in  MATLAB  and run on a computer with 16 GB of RAM and a 3.8 GHz  Intel Core i7  processor. The total simulation time for each example was in the order of 1-5 minutes.

% \subsection{Example 1: Analytical}

% \paragraph{Analytically tractable 4D example.}
% To illustrate the proposed KBK filter with closed-form principal Koopman eigenfunctions, we consider the nonlinear system
% \begin{align}\label{eq:ex4d}
% \dot x_1=\lambda_1 x_1,\qquad
% \dot x_2=\lambda_2 x_2,\nonumber\\
% \dot x_3=-x_3+x_1^2,\qquad
% \dot x_4=-x_4+x_2^2,
% \end{align}
% where $\lambda_1>0$ and $\lambda_2<0$ so that the equilibrium at the origin is hyperbolic.
% For this system, a set of principal Koopman eigenfunctions can be constructed analytically:
% \begin{align}
% \phi_1(x)=x_1,\quad \phi_2(x)=x_2,\nonumber\\
% \phi_3(x)=x_3-a x_1^2,\quad
% \phi_4(x)=x_4-b x_2^2,
% \end{align}
% with $a=(1+2\lambda_1)^{-1}$ and $b=(1+2\lambda_2)^{-1}$, and corresponding eigenvalues
% $\lambda(\phi_1)=\lambda_1$, $\lambda(\phi_2)=\lambda_2$, and $\lambda(\phi_3)=\lambda(\phi_4)=-1$.
% Moreover, the product observables $\psi_5(x)=x_1^2$ and $\psi_6(x)=x_2^2$ are Koopman eigenfunctions with eigenvalues $2\lambda_1$ and $2\lambda_2$, respectively.
% These eigenfunctions yield an explicit lifted coordinate map
% $\Phi(x)=[\phi_1,\phi_2,\phi_3,\phi_4,\psi_5,\psi_6]^\top$ in which the lifted dynamics are linear:
% \[
% \dot \Phi = \Lambda \Phi,
% \]
% for a diagonal $\Lambda$ determined by the above eigenvalues. Finally, with the measurement
% $y=[x_3\;\;x_4]^\top+v$, the output admits a linear representation in Koopman coordinates,
% $y=C_\Phi \Phi(x)+v$, where $C_\Phi$ is known explicitly from $x_3=\phi_3+a\psi_5$ and
% $x_4=\phi_4+b\psi_6$. 

\subsection{Example 1: Analytical} Our first example is one for which the eigenfunctions can be computed analytically. Furthermore, the spanning Assumption \ref{assume_span} is also satisfied and hence the filter equation can be solved analytically. 
\begin{align}
\dot x_1&=\rho x_1+w_1(t)\nonumber\\
\dot x_2&=\mu x_2-(\rho^2-\mu)c x_1^2+w_2(t)\label{example1}\\
y&=x_2-dx_1^2+x_1^2+v(t)\label{example1output}
\end{align}
where $d=\frac{(\mu-\rho^2)c}{\mu-2\rho}$. The Koopman principal eigenfunctions and eigenfunction of this system can be computed analytically and of the form 
\[
\Phi(x)
=
\begin{bmatrix}
\phi_1 \\[4pt] \phi_2
\end{bmatrix}
=
\begin{bmatrix}
x_1 \\[4pt]
x_2-d x_1^2
\end{bmatrix},
\qquad
\Lambda
=
\operatorname{diag}(\rho,\mu).
\]
For this output, the Assumption \ref{assume_span} is not satisfied and hence we have to lift the eigenfunctions to expressed the output in terms of lifted eigenfunctions. In particular, the output can be expressed as 
\[y=\phi_2(x)+\phi_1^2(x)+v(t)\]
The lifted eigenfunctions for this case will be of the form
\[\bar \Phi(x)=\begin{pmatrix}\phi_1(x)\\\phi_2(x)\\\phi_1^2(x)\end{pmatrix}\]

The various parameter values are
\[\rho = -0.1,\;
\mu  = -0.3,\;
c   = 1.0,\; R   = 1e^{-4} I,\;
Q   = 1e^{-1}\]

% \[\rho = -0.05,\;
% \mu  = -0.08,\;
% c   = 1.0,\; R   = 1e^{-4} I,\;
% Q   = 1e^{-1}\]
The system matrices for this case will be 
\[\bar \Lambda={\rm diag}(\rho,\mu,2\rho),\;\;\bar C=\begin{pmatrix}0&1&1\end{pmatrix}\]

We use the dynamics-inspired basis function as discussed in Section \ref{section_basisPI} for the computation of the principal eigenfunctions. For the construction of dynamics-inspired basis functions, we use $441$ data point uniformly sampled in the domain $[-1,1]^2$ with $\Delta=\{0.1,0.2,0.4\}$ the time step of integration (refer to Eq. \ref{stepsintegration}) and $\varepsilon=10^{-5}$ used for approximation of the gradient of eigenfunction.  In Fig. \ref{fig1:example1}, we show the comparison of the error plots between the $\|x_{true}-x_{KBK}\|$ and $\|x_{true}-x_{EKF}\|$. The performance of the proposed KBK filter is comparable to the EKF filter. However, the main advantage of the KBK filter is that it gives complete information of the state density function as shown in Fig. \ref{fig2:example1}.   
 
 \begin{figure}
    \centering
    \includegraphics[width=.8\linewidth]{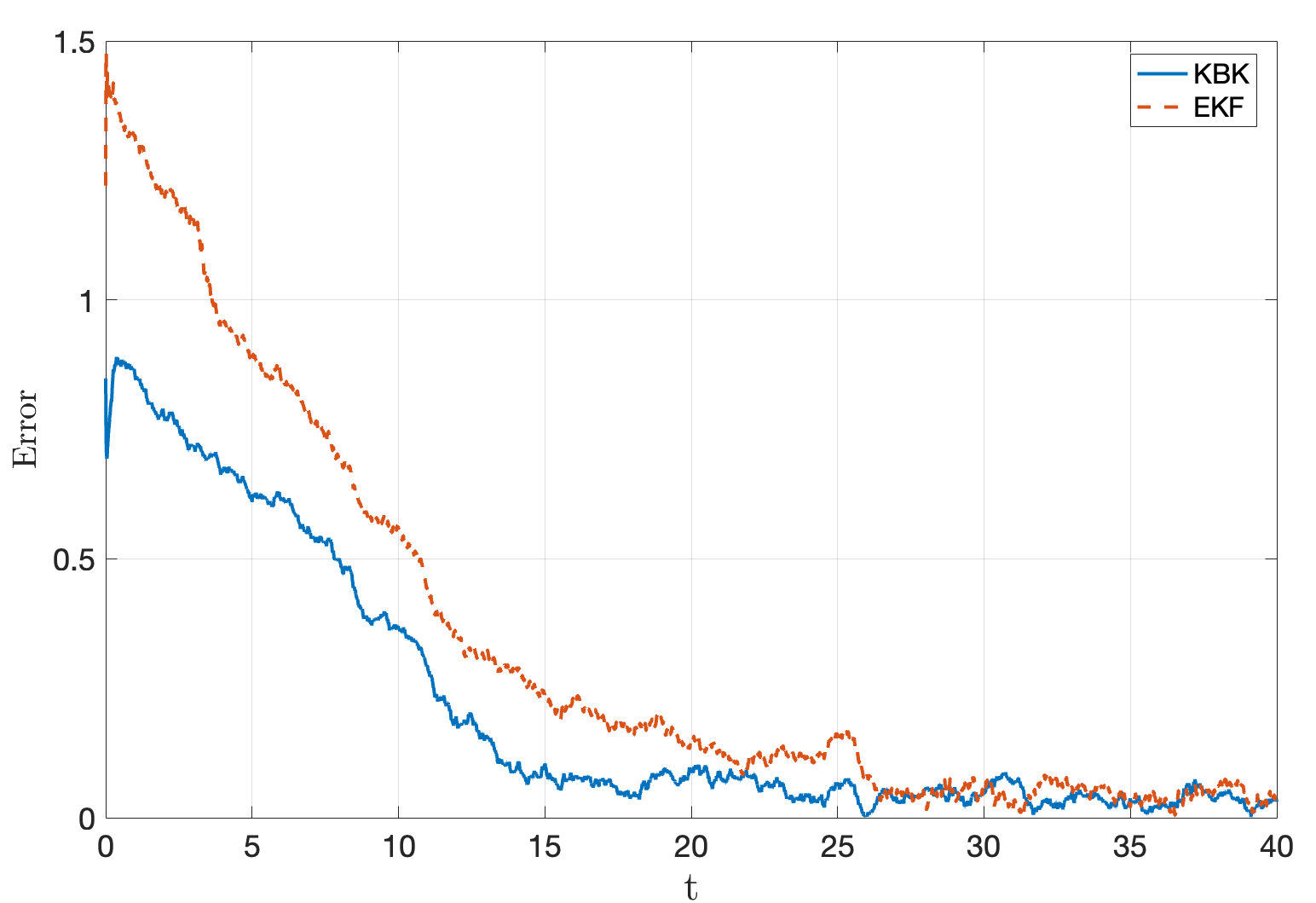}
    \caption{Comparison of the error estimates between the KBK, (i.e., $\|x_{true}-x_{KBK}\|$ ) and EKF filter (i.e., $\|x_{true}-x_{EKF}\|$ ).}
    \label{fig1:example1}
\end{figure}\begin{figure}
    \centering
    \includegraphics[width=1.0\linewidth]{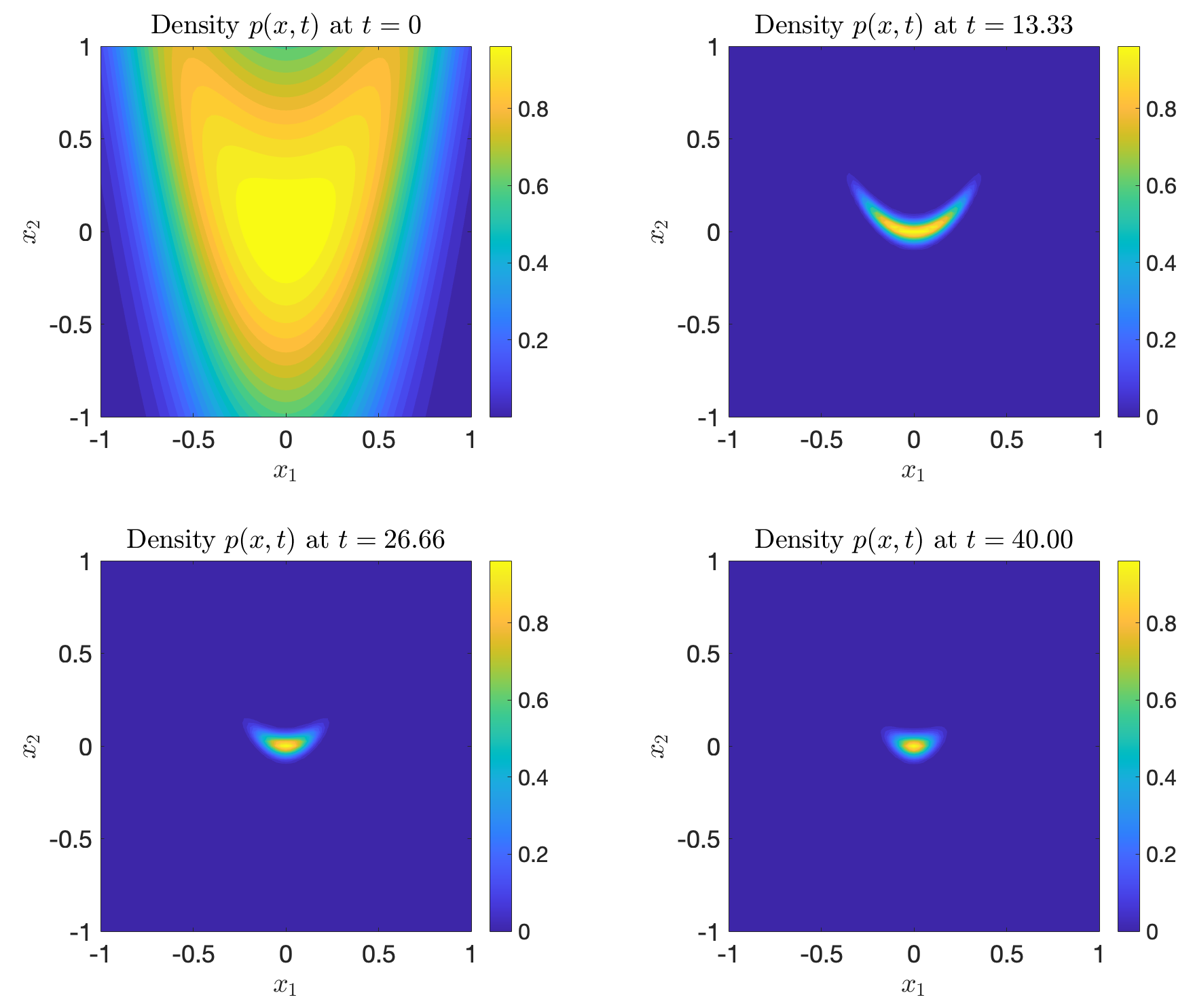}
    \caption{Time evolution of the state density function, $p(x,t)$.}
    \label{fig2:example1}
\end{figure}

% \begin{figure}
%     \centering
%     \includegraphics[width=0.8\linewidth]{fig21.png}
%     \caption{Comparison of the error estimates between the KBK, (i.e., $\|x_{true}-x_{KBK}\|$ ) and EKF filter (i.e., $\|x_{true}-x_{EKF}\|$ ).}
%     \label{fig1:example2}
% \end{figure}\begin{figure}
%     \centering
%     \includegraphics[width=1.0\linewidth]{fig22.png}
%     \caption{Time evolution of the state density function, $p(x,t)$}
%     \label{fig2:example2}
% \end{figure}

% For this example, the eigenfunctions are computed from data using the path integral formula. The equilibrium point for this example is asymptotically stable where the two eigenvalues, $(\rho=-0.05, \mu=-0.08)$ satisfy the spread condition (\ref{eigenspread}) condition and hence the results of Theorem \ref{theorem_pi} apply to compute the eigenfunction using path integral formula. 

% In Fig. \ref{fig2:example1}, we compare the state-estimation errors obtained using the proposed KBK filter and the Extended Kalman Filter (EKF). Specifically, we plot the error norms $\|x_{true}(t)-x_{KBK}(t)\|$ and 
% $\|x_{true}(t)-x_{EKF}(t)\|$ over time. The results show that the KBK filter exhibits superior transient performance relative to the EKF.

The critical difference and the main advantage of the the KBK filter over EKF filter is the ability to compute the state density function. In Fig. \ref{fig2:example1}, we illustrate the evolution of the state density function at several time instants. The density $p(x,t)$ is computed using the relation $p(x,t)=\exp(-V(x,t))$, where 
$V(x,t)$ admits a quadratic parameterization in terms of the principal eigenfunctions of the Koopman operator. As time progresses, the reduction in uncertainty is clearly reflected by the decreasing spread of the density, indicating improved concentration of the posterior estimate. 

\subsection{Example 2: Robust State Estimation for Quadrotor Vertical Dynamics with Parameter Mismatch}
\label{sec:quad_example}
% =========================================================

The primary objective of this example is to demonstrate the inherent robustness of the proposed KBK filter to modeling uncertainty and to contrast this robustness with the sensitivity of the Extended Kalman Filter (EKF) under parameter mismatch. 
In practical quadrotor systems, parameters such as damping coefficients, actuator gains, and sensor nonlinearities are rarely known with high precision. 
Unmodeled aerodynamic effects, actuator saturation, and environmental disturbances further introduce discrepancies between the true system dynamics and the nominal model used for estimation. 

Since the EKF is explicitly model-based, its gain computation relies on the assumed system parameters; in our implementation, the EKF is constructed using mismatched parameters to reflect realistic modeling errors. 
In contrast, the KBK filter computes the spectral representation and associated Riccati evolution using data generated from the true system, thereby avoiding direct reliance on the potentially inaccurate parametric model. 
This setup enables a direct comparison of robustness properties under structured modeling uncertainty.

Let
$x = \begin{bmatrix} z & v & \xi \end{bmatrix}^\top$
denote altitude $z$, vertical velocity $v=\dot z$, and an internal aerodynamic/actuator state $\xi$ that captures unmodeled nonlinear effects such as propeller wake interactions or actuator saturation.
\begin{align}
\dot z &= v+w_1(t), \nonumber\\
\dot v &= -k z - \theta_d c_{\text{base}} v + g(\xi) + b u+w_2(t), \nonumber\\
\dot \xi &= -\alpha \xi+w_3(t).
\label{eq:quad_model}
\end{align}
where,  $k>0$ represents local linearization of gravity about hover, $\theta_d c_{\text{base}}$ is an uncertain aerodynamic damping coefficient,  $g(\xi)=1.2 \tanh(\xi), 
$ models nonlinear internal dynamics,
$u$ is the thrust control input. Two key parameters are assumed unknown: $\theta_d$ (damping scaling),  $\theta_m$ in the nonlinear measurement model
\[
y = \tanh(\theta_m z) + v(t),
\]
representing sensor saturation or nonlinear altitude sensing.
This setting captures realistic mismatch in both dynamics and sensing.
The true parameters are:
\[
k=2.0, \quad c_{\text{base}}=2.0, \quad \alpha=3.0, \quad b=1.0,
\]
\[
\theta_d^{\text{true}}=1.6, \qquad \theta_m^{\text{true}}=1.0.
\]
 The process noise intensity:
$
R=\operatorname{diag}(0.02^2,\,0.06^2,\,0.05^2)
$ and the 
measurement noise variance:
$
Q= (0.03)^2.
$
The input excitation is
\[
u(t)=2\sin(0.5 t)+1.4\sin(1.6 t+0.7),
\]
which drives the system through nonlinear operating regimes. For the KBK filter design all the quantities were computed from the data. In particular, the eigenfunctions $\Phi(x)$ and the eigenmatrix $\Lambda$ are computed from data. 
% (i) identifying the dominant linear dynamics of the $(z,v)$ subsystem,
% (ii) approximating the nonlinear coupling term $g(\xi)$ from data, and
% (iii) computing the nonlinear correction to the eigenfunctions using a
% finite-horizon path-integral formula.

% \paragraph{Simulation Data.}
The true stochastic system is simulated over a horizon
$T = 60~\text{s}$,
with time step $
dt = 2\times 10^{-3}~\text{s}$. 
The total number of samples is
$N = \frac{T}{dt} = 30000.
$. 
Time derivatives required for regression are computed using central
finite differences, $
\dot{x}_k \approx \frac{x_{k+1}-x_{k-1}}{2dt},$
with forward/backward differences at the boundaries.
The velocity dynamics are modeled as
\[
\dot v = -k z - c v + b u + g(\xi),
\]
where $g(\xi)$ is unknown. The regression model uses the basis
$
\{-z,\,-v,\,u,\,\xi,\,\tanh(\xi)\}.
$ Thus we fit
\[\dot v \approx
k_{\mathrm{est}}(-z)
+
c_{\mathrm{est}}(-v)
+
b_{\mathrm{est}} u
+
a_\xi \xi
+
a_{\tanh}\tanh(\xi).
\]
This yields $
g_{\mathrm{hat}}(\xi)
=
a_\xi \xi
+
a_{\tanh}\tanh(\xi),
$.
% The identified linear subsystem matrix is
% \[
% A_2^{\mathrm{data}}
% =
% \begin{bmatrix}
% 0 & 1\\
% -k_{\mathrm{est}} & -c_{\mathrm{est}}
% \end{bmatrix}.
% \]
% \paragraph{Data-Driven Spectrum.}
The Koopman principal eigenvalues estimated from data equals
$
\lambda_1 = -0.8845,\;\;
\lambda_2 = -2.2435,\;\;\lambda_3=-3.013
$
which matches closely with the eigenvalues of the linearization of the system dynamics at the origin $\lambda_1=-0.8517,\lambda_2=   -2.3483,\;\;\lambda_3=-3$. 
The Koopman eigenfunctions are constructed as
\begin{equation}
\phi_i(x)
=
w_i^\top
\begin{bmatrix}
z \\
v
\end{bmatrix}
+
h_{\lambda_i}(\xi),
\qquad i=1,2, \;\;\phi_3(x)=\xi,
\end{equation}
The nonlinear correction term is computed using a finite-horizon
approximation of the path integral
\begin{equation}
h_\lambda(\xi)
=
\int_0^\infty
e^{-\lambda t}
g(\xi e^{-\alpha t})
\,dt.
\end{equation}
Numerically, we approximate this integral over a finite horizon
of 
$T_h = 4~\text{s}.$ The numerical quadrature parameters are:
$ dt_{\mathrm{int}} = 10^{-3}~\text{s},
n_{\mathrm{int}} = 4000.
$. 
Thus the approximation is
\begin{equation*}
h_\lambda(\xi)
\approx
\sum_{k=0}^{n_{\mathrm{int}}-1}
e^{-\lambda t_k}
g_{\mathrm{hat}}(\xi e^{-\alpha t_k})
\,dt_{\mathrm{int}},
\qquad
t_k = k\,dt_{\mathrm{int}}.
\end{equation*}

% The derivative used in Jacobian computations is similarly approximated as
% \begin{equation}
% \frac{d}{d\xi}h_\lambda(\xi)
% \approx
% \sum_{k=0}^{n_{\mathrm{int}}-1}
% e^{-\lambda t_k}
% g'_{\mathrm{hat}}(\xi e^{-\alpha t_k})
% e^{-\alpha t_k}
% \,dt_{\mathrm{int}}.
% \end{equation}

% \paragraph{Summary of Eigenfunction Approximation Parameters.}

% \begin{itemize}
% \item Training data length: $N_{\mathrm{tr}} = 30000$ samples.
% \item Sampling time: $dt = 0.002$ s.
% \item Regression basis for $g(\xi)$: $\{\xi,\tanh(\xi)\}$.
% \item Identified spectrum from $A_2^{\mathrm{data}}$.
% \item Path-integral horizon: $T_h = 4$ s.
% \item Quadrature step: $dt_{\mathrm{int}} = 10^{-3}$ s.
% \item Number of quadrature points: $n_{\mathrm{int}} = 4000$.
% \end{itemize}
% For simulation, we use 
The eigenfunction residual RMS errors are as follows.
\begin{align*}
\mathrm{RMS}_{\lambda_1} &= 1.742\times10^{-2},\;\;
\mathrm{RMS}_{\lambda_2} = 3.774\times10^{-2},\nonumber\\
\mathrm{RMS}_{\xi} &= 7.946\times10^{-14}.
\end{align*}

We project the nonlinear measurement into Koopman coordinates, $ h(x)\approx C\Phi(x)$.  Least-squares yields
$
C=
\begin{bmatrix}
1.4919 & -0.8103 & 0.4857
\end{bmatrix},
$
with projection RMS error $4.447\times10^{-2}$.

The EKF is designed with incorrect parameters, $
\theta_d^{0}=1.2, \;\; \theta_m^{0}=0.6.
$ as the EKF is designed based on the analytical system model and its linearization. 
Thus, both the dynamics and measurement models are inaccurate. In Figs. \ref{fig1:example2} and \ref{fig2:example2}, we compare the performance of the EKF and the proposed KBK filter against parameter mismatch. In Fig. \ref{fig1:example2}, we notice that the KBK error is close to zero whereas the state error for the EKF is nonzero. In Fig. \ref{fig2:example2}, we vary the EKF assumed $\theta_m^0$ and plot the aggregate root mean square error vs $\theta_m^0$. 
The EKF error increases rapidly with mismatch, while KBK remains essentially invariant. Note that the true value of the $\theta^{true}_m=1$, where the EKF RMSE is the smallest. This example and the simulation results highlights the
sensitivity of EKF to parametric model errors and the robustness of KBK under parametric uncertainty.

% \begin{figure}
%     \centering
%     \includegraphics[width=1.0\linewidth]{fig1quad.png}
%     \caption{Comparison of the KBK  with the EKF filter error plots. }
%     \label{fig2:example2}
% \end{figure}

\begin{figure}
    \centering
    \includegraphics[width=1.0\linewidth]{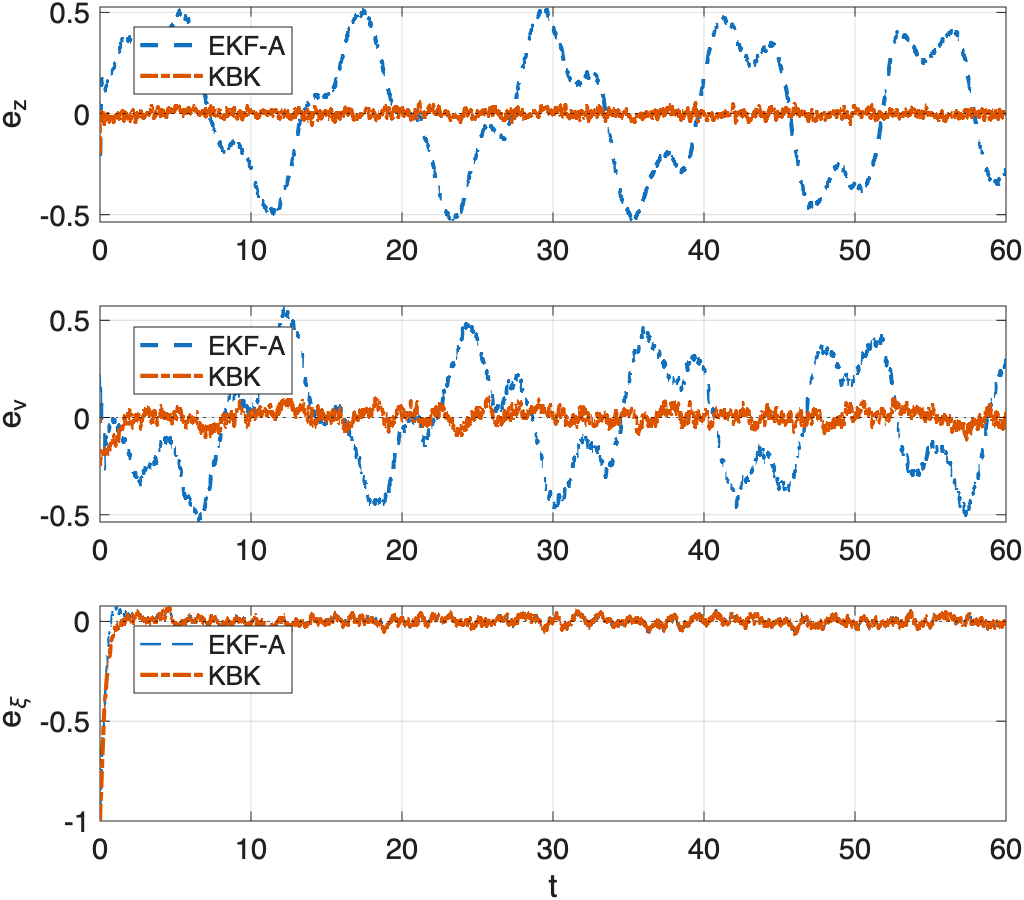}
    \caption{Comparison of the KBK  with the EKF filter: State error plots. }
    \label{fig1:example2}
\end{figure}

\begin{figure}
    \centering
    \includegraphics[width=0.8\linewidth]{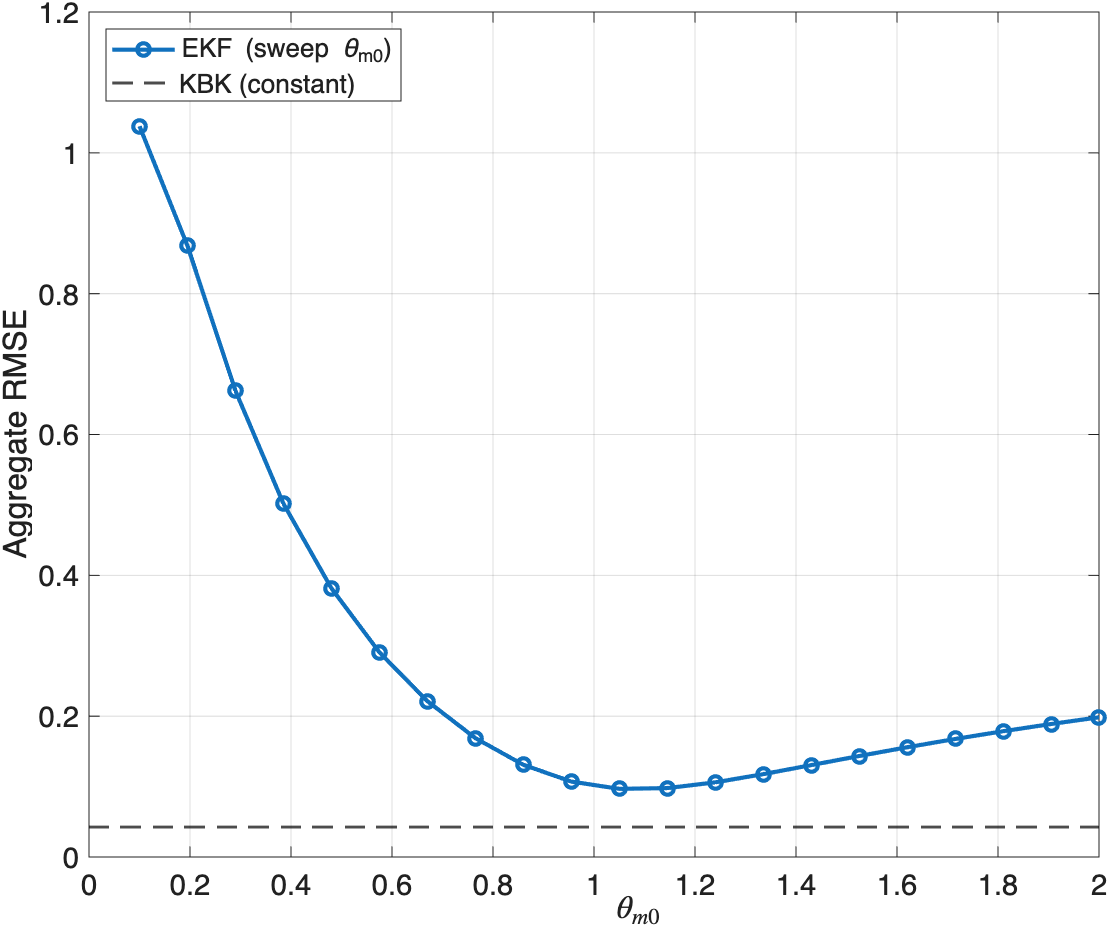}
    \caption{Comparison of the KBK  with the EKF filter}
    \label{fig2:example2}
\end{figure}

\subsection{Example 3: Data-driven estimator for vehicle dynamics}  

Our third example considers a controlled dynamical system motivated by off-road vehicle applications and is designed to illustrate the data-driven capability of the proposed KBK filter. Although an analytical model is used to generate simulation data, the KBK filter itself does not require explicit knowledge of the system vector field for implementation. In particular, the principal Koopman eigenfunctions and the associated filter gain are computed directly from trajectory data using the proposed spectral and path-integral procedures, without linearization or access to the system Jacobian. This contrasts with classical approaches such as the extended Kalman filter, which rely explicitly on model-based linearization. The example therefore demonstrates that the KBK framework can operate as a purely data-driven estimator while retaining the structural advantages of Riccati-type filtering.

Accurate estimation of the tire road friction coefficient $\mu$ is essential for stability control, autonomous driving, and safety-critical decision making. The friction coefficient $\mu$ cannot be measured directly and enters the vehicle dynamics only through nonlinear tire forces, it must be inferred from measurable motion variables such as lateral acceleration and yaw rate. The estimation problem is challenging due to the nonlinear dependence of $F_{y,f}$ and $F_{y,r}$ on slip angles, the weak observability of $\mu$ under mild steering excitation, and the possibility of abrupt changes in road conditions that require fast yet stable adaptation.

The dynamic model used for estimation is a reduced lateral--yaw bicycle model with constant forward speed $u_0$, where the state vector is $x = [\,v,\ r,\ \mu\,]^\top$. The evolution equations are
\begin{align}
    \dot{v} &= \frac{1}{m}\left(F_{y,f} + F_{y,r}\right) - u_0 r,\;\;\;
    \dot{r} &= \frac{1}{I_z}\left(l_f F_{y,f} - l_r F_{y,r}\right),\nonumber \\
    \dot{\mu} &= w(t).\label{vehicle_dynamics}
\end{align}
where $w\in \mR^3$ is the white Gaussian noise. To define the tire forces, we first define the slip angles as
\begin{align}
    \alpha_f &= \delta - \frac{v + l_f r}{u_0}, \qquad
    \alpha_r = -\frac{v - l_r r}{u_0},
\end{align}
and the nonlinear lateral forces follow the Pacejka form
\begin{align}
    F_{y,f} &= \mu F_{z,f}\,
        \sin\!\left(C\arctan\!\big(B\alpha_f - E (B\alpha_f - \arctan(B\alpha_f))\big)\right)\nonumber \\
    F_{y,r} &= \mu F_{z,r}\,
        \sin\!\left(C\arctan\!\big(B\alpha_r - E (B\alpha_r - \arctan(B\alpha_r))\big)\right)
\end{align}
with static normal loads
$
F_{z,f} = \frac{m g l_r}{l_f + l_r},\;\;\;
F_{z,r} = \frac{m g l_f}{l_f + l_r}.
$. 
The measurement vector consists of the lateral acceleration and yaw rate,
\[
y = \begin{bmatrix} a_y \\ r \end{bmatrix}+v(t)=h(x)+v(t) \qquad
a_y = \frac{F_{y,f} + F_{y,r}}{m}.
\]
Since $a_y$ depends on $\mu$ through nonlinear tire dynamics, friction can be estimated from these outputs provided sufficient steering excitation. The various parameter values used in the simulation are as follows. 
\begin{align}
m &= 1500~\text{kg},\;\; I_z  3000~\text{kg}\cdot\text{m}^2, \;\;
l_f = 1.2~\text{m}\nonumber \\l_r &= 1.6~\text{m},\;\; 
u_0 = 10~\text{m/s}, \;\; g = 9.81~\text{m/s}^2,\nonumber \\
B &= 10.0, \;\; C = 1.9, \;\; E = 0.97,\nonumber \\
F_{z,f} &= \frac{m g l_r}{l_f + l_r} = 8829~\text{N}, \;\;
F_{z,r} = \frac{m g l_f}{l_f + l_r} = 6622~\text{N}.\nonumber
\end{align}
The measurement and noise covariance matrix and excitation steering input are chosen to be 
\[
Q= 4\mathrm{diag}(0.1^2,\ 0.01^2),\;\;\;\delta(t)=4
\sin(1.2 \pi t)
\]
No process noise was added except for the small noise to the friction coefficient state. The Koopman-based lifting model is constructed directly from simulation data generated across multiple friction conditions.  We simulate the nonlinear model (\ref{vehicle_dynamics}) for four friction levels $\mu \in \{0.4,0.6,0.8,1.0\}$, each over an $8$\,s horizon with $\Delta t = 0.01$\,s, producing approximately $3\times 10^4$ state-transition samples for DMD.  These snapshot pairs are used to identify a global linear operator $A_{\mathrm{DMD}}$ whose real-valued eigenvectors define the three linear Koopman coordinates, $\Phi=W^\top x$.  Quadratic products of these coordinates form a nine-dimensional extended eigenfunction basis used for lifting.  The nonlinear measurement map $y=h(x)$ is similarly approximated from data by solving a least-squares regression $y \approx \bar C\,\bar \Phi(x)$ over all training trajectories.  Together, the learned lifting map and output operator provide a fully data-driven representation used for computing the KBK gain in the original state space.

The friction coefficient in the simulation is varied from $\mu = 0.3$ to $\mu = 0.7$ halfway through the experiment to evaluate estimator performance under abrupt road condition changes. In Fig. \ref{fig1:example3}, we show the performance of the KBK filter with the varying friction coefficient. From this plot it is clear that the filter can respond to sudden changes in the friction and estimate closely follows the true friction value demonstrating the effecitiveness of data-driven filter design. 

\begin{figure}
    \centering
    \includegraphics[width=0.8\linewidth]{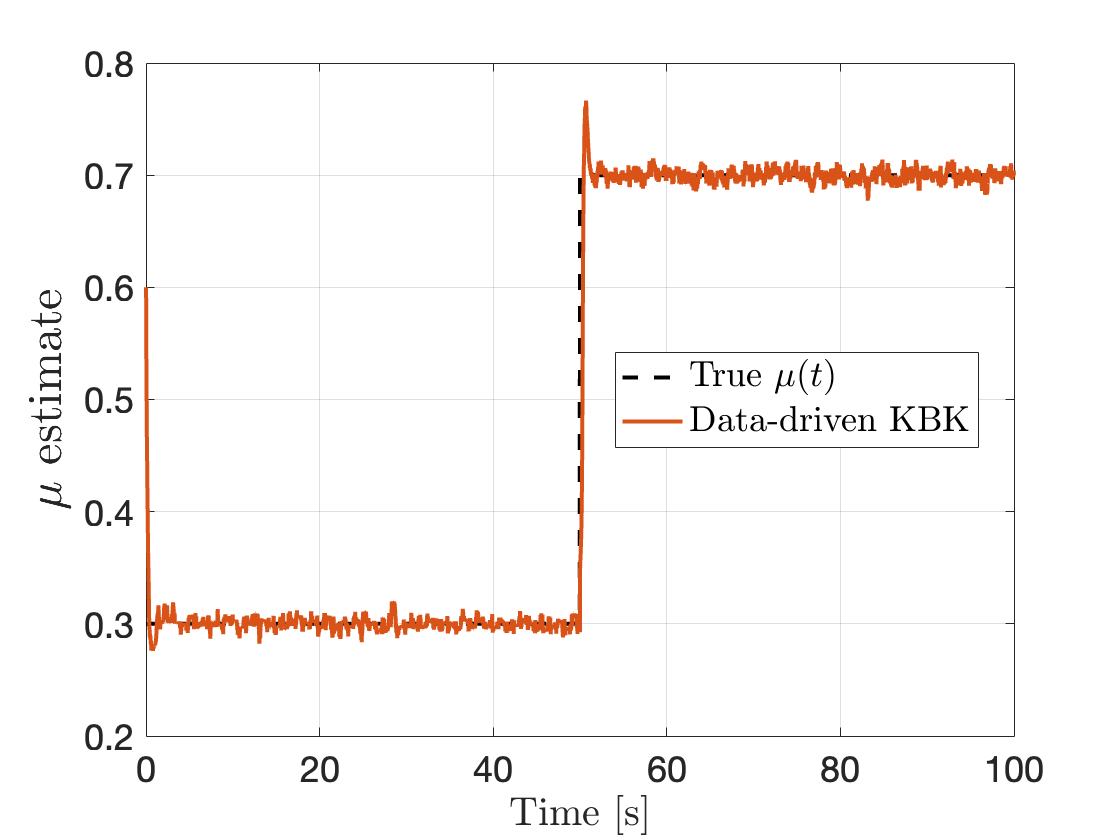}
    \caption{KBK filter for the estimation of friction coefficient}
    \label{fig1:example3}
\end{figure}

\section{Conclusion}
In this paper, we introduced a new framework for nonlinear state estimation based on the spectral properties of the Koopman operator. By lifting the nonlinear dynamics into the Koopman eigenfunction coordinates, we transformed the nonlinear estimation problem into a linear estimation problem in the Koopman spectral space. We showed that the optimal value function for the maximum-likelihood estimator admits a quadratic representation in terms of the principal Koopman eigenfunctions, and that the negative exponential of this value function yields a state-density representation that naturally captures estimation uncertainty.
To enable the practical computation of the principal eigenfunctions, we proposed both a path-integral formulation and a class of dynamics-informed basis functions derived from the method of characteristics. These constructions provide principled mechanisms for approximating the spectral structure governing the estimator.
Future work will focus on developing numerically efficient algorithms for computing Koopman eigenfunctions using these path-integral and characteristic-based techniques, as well as addressing limitations that arise from the non-global nature of principal eigenfunctions in filter design. Together, these directions aim to extend the applicability and robustness of the proposed Koopman-based estimation framework.

\bibliographystyle{IEEEtran}
\bibliography{references_clean}
\begin{IEEEbiography}[{\includegraphics[width=1in,height=1.25in,clip,keepaspectratio]{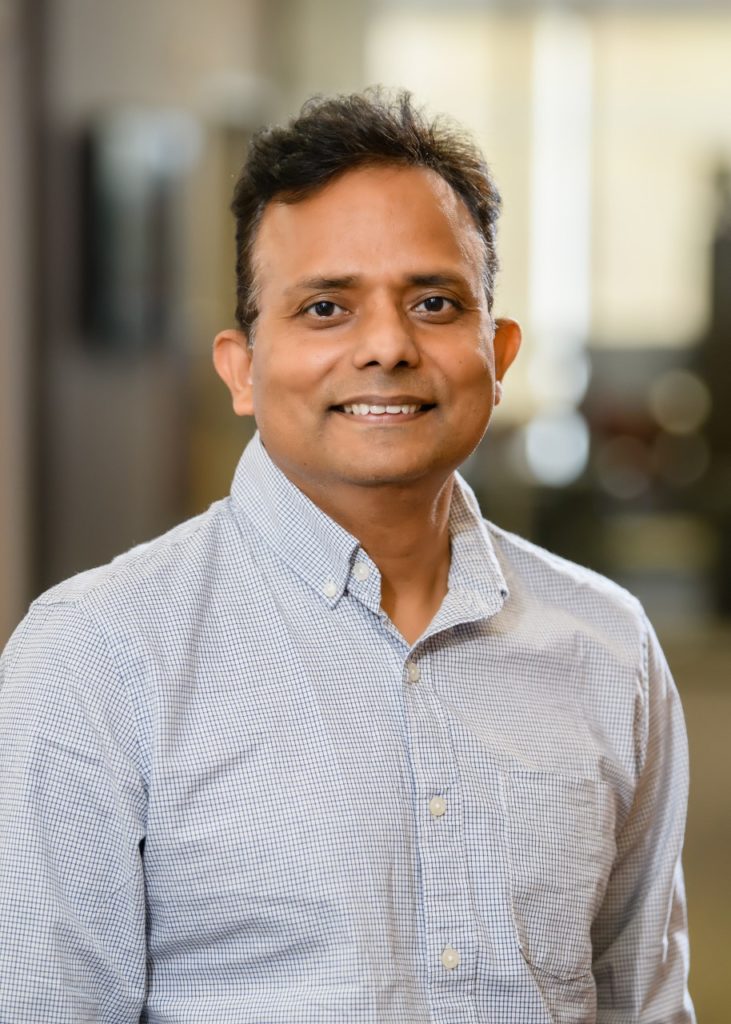}}]%
{Umesh Vaidya}(M’07, SM'19)  received the Ph.D. degree in mechanical engineering from the University of California at Santa Barbara, Santa Barbara, CA, in
2004. He was a Research Engineer at the United Technologies Research Center (UTRC), East Hartford, CT, USA. He is currently a professor in the Department of Mechanical Engineering, Clemson University, S.C., USA. Before joining Clemson University in 2019, and since 2006, he was a faculty with the department of Electrical and Computer Engineering at Iowa State University. He is the recipient of 2012 National Science Foundation CAREER award. His current research interests include dynamical systems and control theory with applications to power grid and robotics. 
\end{IEEEbiography}
\end{document}